\documentclass[aos,MSNbibl,seceqn,dvips]{arximspdf}
\usepackage{graphicx}
\doi{10.1214/13-AOS1155} 
\volume{41}
\issue{5}
\pubyear{2013}
\firstpage{2428}
\lastpage{2461}

\makeatletter

\newcommand{\iint}{\int\!\!\int}

\newtheorem{theorem}{Theorem}[section]
\newtheorem{lem}[theorem]{Lemma}

\newcommand{\ep}{\varepsilon}
\newcommand{\mg}{\mathcal{G}}
\newcommand{\mm}{\mathcal{W}}
\newcommand{\mmm}{\widetilde{\mathcal{W}}}
\newcommand{\pp}{\mathbb{P}}
\newcommand{\rr}{\mathbb{R}}
\newcommand{\tilf}{{\widetilde{F}}}
\newcommand{\tilh}{{\widetilde{h}}}
\newcommand{\tilu}{{\widetilde{U}}}

\newcommand{\cov}{\operatorname{Cov}}
\newcommand{\ee}{\mathbb{E}}

\newcommand{\ra}{\rightarrow}

\newcommand{\WF}{{\widetilde F}}
\newcommand{\wh}{{\widetilde h}}

\makeatother

\begin{document}
\begin{frontmatter}

\title{Estimating and understanding exponential random graph models}
\runtitle{Exponential random graph models}

\begin{aug}
\author[A]{\fnms{Sourav} \snm{Chatterjee}\corref{}\thanksref{t1}\ead[label=e1]{souravc@stanford.edu}}
\and
\author[A]{\fnms{Persi} \snm{Diaconis}\thanksref{t2}\ead[label=e2]{diaconis@math.stanford.edu}}
\runauthor{S. Chatterjee and P. Diaconis}
\affiliation{Stanford University}
%
\address[A]{Department of Statistics\\
Stanford University\\
Sequoia Hall\\
390 Serra Mall\\
Stanford, California 94305\\
USA\\
\printead{e1}\\
\hphantom{E-mail: }\printead*{e2}}
\end{aug}

\thankstext{t1}{Supported in part by NSF Grant DMS-10-05312 and a Sloan
Research Fellowship.}

\thankstext{t2}{Supported in part by NSF Grant DMS-08-04324.}

\received{\smonth{3} \syear{2011}}
\revised{\smonth{7} \syear{2013}}

%
\begin{abstract}
We introduce a method for the theoretical analysis of exponential
random graph models. The method is based on a large-deviations
approximation to the normalizing constant shown to be consistent using
theory developed by Chatterjee and Varadhan [\textit{European J.
Combin.} \textbf{32} (2011) 1000--1017]. The theory explains a host of
difficulties encountered by applied workers: many distinct models have
essentially the same MLE, rendering the problems ``practically''
ill-posed. We give the first rigorous proofs of ``degeneracy'' observed
in these models. Here, almost all graphs have essentially no edges or
are essentially complete. We supplement recent work of Bhamidi, Bresler
and Sly [\textit{2008 IEEE 49th Annual IEEE Symposium on Foundations of
Computer Science} (\textit{FOCS}) (2008) 803--812 IEEE] showing that
for many models, the extra sufficient statistics are useless: most
realizations look like the results of a simple Erd\H{o}s--R\'enyi
model. We also find classes of models where the limiting graphs differ
from Erd\H{o}s--R\'enyi graphs. A limitation of our approach, inherited
from the limitation of graph limit theory, is that it works only for
dense
graphs. 
\end{abstract}

%

\pdfkeywords{62F10, 05C80, 62P25, 60F10, Random graph,
Erdos--Renyi, graph limit,
exponential random graph models,
parameter estimation}

\end{frontmatter}

\section{Introduction}\label{sec1}

Graph and network data are increasingly common and a host of
statistical methods have emerged in recent years. Entry to this large
literature may be had from the research papers and surveys in Fienberg
{\cite{fienberg,fienberg2}}. One mainstay of the emerging theory are
the exponential families
%
%
\begin{equation}\label{11}
p_\beta(G)=\exp\Biggl(\sum_{i=1}^k
\beta_iT_i(G) - \psi(\beta) \Biggr),
\end{equation}
where $\beta=(\beta_1,\ldots,\beta_k)$ is a vector of real parameters,
$T_1,T_2,\ldots,T_k$ are functions on the space of graphs (e.g., the
number of edges, triangles, stars, cycles$,\ldots$), and
$\psi$ is the normalizing constant. In this paper, $T_1$ is usually
taken to be the number of edges (or a constant multiple of it). 

%
%
\begin{figure}

\includegraphics{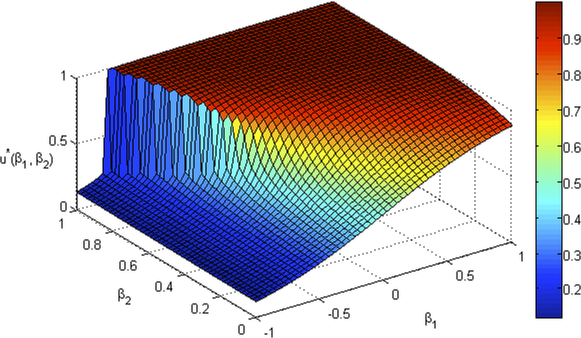}

\caption{The plot of $u^*$ against $(\beta_1,\beta_2)$. There is a
discontinuity on the left where $u^*$ jumps from near $0$ to near $1$;
this corresponds to a phase transition. (Picture by Sukhada Fadnavis.)}
\label{fig1}
\end{figure}

We review the literature of these
models in Section \ref{sec2a}. Estimating the parameters in these
models has
proved to be a challenging task. First, the normalizing constant
$\psi(\beta)$ is unknown. Second, very different values of $\beta$ can
give rise to essentially the same distribution on graphs.

Here is an example: consider the model on simple graphs with $n$ vertices,
%
%
\begin{equation}
\label{trianglemodel} p_{\beta_1,\beta_2}(G)=\exp\biggl(2\beta_1 E+
\frac{6\beta
_2}{n}\Delta-n^2\psi_n(\beta_1,
\beta_2) \biggr),
\end{equation}
where $E$, $\Delta$ denote the number of edges and triangles in the
graph $G$. The normalization of the model ensures nontrivial large
$n$ limits. Without scaling, for large $n$, almost all graphs are
empty or full. This model is studied by Strauss \cite{strauss}, Park
and Newman \cite{park04,park05}, H\"aggstrom and Jonasson \cite
{hj99}, and many others.

Theorems \ref{soln} and \ref{specialnorm} will show that for $n$ large and
nonnegative $\beta_2$,
%
%
\begin{equation}\label{12}\quad
\psi_n(\beta_1, \beta_2)\simeq\sup
_{0\leq u\leq1} \biggl(\beta_1u+\beta_2u^3-
\frac12 u\log u-\frac12(1-u)\log(1-u) \biggr).
\end{equation}
The maximizing value of the right-hand side is denoted
$u^*(\beta_1,\beta_2)$. A plot of this function appears in Figure \ref
{fig1}. Theorem \ref{specialbehave} shows that for any $\beta_1\in\rr$ and
$\beta_2>0$, with
high probability, a pick from $p_{\beta_1,\beta_2}$ is essentially the
same as an Erd\H{o}s--R\'enyi graph generated by including edges
independently with
probability $u^*(\beta_1,\beta_2)$. This phenomenon has previously
been identified by Bhamidi et al. \cite{bhamidi} and is
discussed further in Section \ref{sec2a}. Figure \ref{fig2} shows the
contour lines for Figure \ref{fig1}. All the $(\beta_1,\beta_2)$ values on
the same contour line lead to the same Erd\H{o}s--R\'enyi model in the limit.
%

%
%
\begin{figure}

\includegraphics{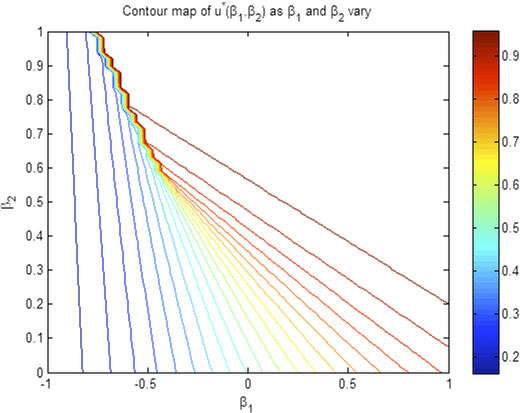}

\caption{Contour lines for Figure \protect\ref{fig1}. All pairs $(\beta
_1,\beta_2)$
on the same contour line correspond to the same value of $u^*$ and
hence those models will correspond to the same Erd\H{o}s--R\'enyi model
in the
limit. The phase transition region is seen in the upper left-hand
corner where all contour lines converge. (Picture by Sukhada Fadnavis.)}
\label{fig2}
\end{figure}

Our development uses the emerging tools of graph limits as developed
by Lov\'asz and coworkers. We give an overview in Section \ref{sec2c}.
Briefly, a sequence of graphs $G_n$ converges to a limit if the
proportion of edges, triangles, and other small subgraphs in $G_n$
converges. There is a limiting object and the space of all these
limiting objects serves as a useful compactification of the set of all
graphs. Our theory works for functions $T_i(G)$ which are continuous
in this topology. In their study of the large deviations of Erd\H
{o}s--R\'enyi random
graphs, Chatterjee and Varadhan \cite{cv10} derived the associated rate
functions in the language of graph limit theory. Their work is crucial
in the present development and is reviewed in Section~\ref{sec2d}.

Our main results are in Section \ref{sec3} through Section \ref
{sec7}. These sections contain only the statements of the theorems; all
proofs are given in Section~\ref{secproof}.

Working with general exponential models, Section \ref{sec3} gives an
extension of the approximation (\ref{12}) for $\psi_n$ (Theorem \ref{soln})
and shows that, in
the limit, almost all graphs from the model (\ref{11}) are close to
graphs where a certain functional is maximized. As will emerge, sometimes
this maximum is taken on at a unique Erd\H{o}s--R\'enyi model.

The main statistical motivation of this paper comes from the formula
for the limit of the normalizing constant given in Theorem \ref{soln}, since
the normalizing constant is crucial for the computation of maximum
likelihood estimates. At present, the computational tools used by
practitioners to compute the normalizing constants in exponential
random graph models become prohibitively time-consuming even for
moderately large $n$. The theory initiated in this paper hopes to
circumvent this problem by providing analytical formulas. As mentioned
in the abstract, the limitation of our approach is that as of now, it
applies only to dense graphs.

Incidentally, in a recent meeting at the American Institute of
Mathematics, computer-intensive calculations carried out by Mark
Handcock and David Hunter indicated that the formula given in Theorem
\ref
{soln} is actually a pretty good approximation to the exact value of
the normalizing constant even for $n$ as small as 20.

Section \ref{sec4} studies the problem for the model (\ref{11}) when
$\beta_2,\ldots,\beta_k$ are
positive ($\beta_1$ may have any sign). When the $T_i$'s are subgraph
counts, positive $\beta_i$'s were originally intentioned (e.g., in
\cite{frank}) to ``encourage'' the presence of the corresponding
subgraphs. It is shown that the
large-deviations approximation for $\psi_n$ can be easily calculated
as a one-dimensional maximization (Theorem \ref{specialnorm}). Further,
amplifying
the results of Bhamidi et al. \cite{bhamidi}, it is shown that
in these cases, almost all realizations of the model (\ref{11}) are
close to an Erd\H{o}s--R\'enyi graph (or perhaps a finite mixture of
Erd\H{o}s--R\'enyi graphs)
(Theorem \ref{specialbehave}). These mixture cases actually occur for natural
parameter values. Section \ref{sec5} also gives a
careful account of the phase transitions and near-degeneracies
observed in the edge-triangle model (\ref{12}).


Sections \ref{sec6} and \ref{sec7} investigate cases where $\beta_i$
is allowed to be negative. While the general case remains open (and
appears complicated), in Section \ref{sec6} it is shown that Theorems
\ref
{specialnorm} and \ref{specialbehave} hold as stated if $(\beta_i)_{2
\le i \le k}$ are sufficiently small in magnitude. This requires a
careful study of associated Euler--Lagrange equations. Section \ref
{sec7} shows how the results change for the model containing edges and
triangles when $\beta_2$ is negative. For sufficiently large negative
$\beta_2$, typical realizations look like a random bipartite graph
(where ``random'' means that the two parts, of equal size, are chosen
uniformly at random from all possible choices). This is very different
from the Erd\H{o}s--R\'enyi model. The result generalizes to other
models via an interesting analogy with the Erd\H{o}s--Stone theorem
from extremal graph theory. 

A longer version of this paper with more pictures and additional
results is available as ``version 3'' on arXiv
(\url{http://arxiv.org/pdf/1102.2650v3.pdf}).

\section{Background}\label{sec2}

This section gives needed background and notation in three
areas. Exponential graph models (Section \ref{sec2a}), graph limits
(Section~\ref{sec2c}), and large deviations (Section \ref{sec2d}).

\subsection{Exponential random graphs}\label{sec2a}

Let $\mg_n$ be the space of all simple graphs on $n$ labeled
vertices (``simple'' means undirected, with no loops or multiple
edges). Thus,\vadjust{\goodbreak} $\mg_n$
contains $2^{{n\choose2}}$ elements. A variety of models in active
use can be presented in exponential form
%
%
\begin{equation}\label{21}
p_\beta(G)=\exp\Biggl(\sum_{i=1}^k
\beta_iT_i(G)-\psi(\beta) \Biggr),
\end{equation}
where $\beta=(\beta_1,\ldots,\beta_k)$ is a vector of real parameters,
$T_1,T_2,\ldots,T_k$ are real-valued functions on $\mg_n$, and
$\psi(\beta)$ is a normalizing constant. Usually, $T_i$ are taken to
be counts of various subgraphs, for example, $T_1(G)=\#$ edges in $G$,
$T_2(G)=\#$ triangles in $G,\ldots\,$. The main results of
Section \ref{sec3} work for more general ``continuous functions'' on graph
space, such as the degree sequence or the eigenvalues of the adjacency
matrix. This allows models with sufficient statistics of the form
$\sum_{i=1}^n\beta_id_i(G)$ with $d_i(G)$ the degree of vertex
$i$. See, for example, \cite{pd177}.

These exponential models were used by Holland and Leinhardt \cite
{holland} in the directed case. Frank and Strauss \cite{frank}
developed them, showing that if $T_i$ are chosen as edges, triangles,
and stars of various sizes, the resulting random graph edges form a
Markov random field. A general development is in
Wasserman and Faust~\cite{wasserman}. Newer developments, consisting
mainly of new sufficient statistics and new ranges for parameters that
give interesting and practically relevant structures, are
summarized in Snijders et al. \cite{snijders06}. Finally,
Rinaldo et al. \cite{rinaldo} develop the geometric theory for
this class of models with extensive further references.

A major problem in this field is the evaluation of the constant $\psi
(\beta)$ which is crucial for carrying out maximum
likelihood and Bayesian inference. As far as we know, there is no
feasible analytic method for approximating $\psi$ when $n$ is large.
Physicists have tried the technique of mean-field approximations; see
Park and Newman \cite{park04,park05} for the case where $T_1$ is
the number of edges and $T_2$ is the number of two-stars 
or the number of triangles. Mean-field
approximations have no rigorous foundation, however, and are known to
be unreliable in related models such as spin glasses \cite
{talagrand03}. For exponential graph models,
Chatterjee and Dey \cite{chatterjeedey10} prove that they
work for some restricted ranges of $\{\beta_i\}$: values where the
graphs are shown to be essentially Erd\H{o}s--R\'enyi graphs (see
Theorem \ref{specialbehave} below
and \cite{bhamidi}).

A host of techniques for approximating the normalizing constant using
various Monte Carlo schemes have been proposed. 
These include the MCMLE procedure of
Geyer and Thompson \cite{gt92}. 
The bridge sampling approach
of Gelman and Meng \cite{gelman} also builds on techniques
suggested by physicists to estimate free energy [$\psi(\beta)$ in our
context]. The equi-energy sampler of Kou et al. \cite{kou} can
also be harnessed to estimate $\psi$.

Alas, at present writing these procedures seem useful only for
relatively small graphs. For bigger graphs, the run-time of the
Monte Carlo algorithms become unpleasantly long.
Snijders \cite{snijders09} and Handcock \cite{handcock} demonstrate
this empirically with further discussion in \cite{snijders06}. One
theoretical explanation for the poor performance of these techniques
comes from the work of Bhamidi et al. \cite{bhamidi}. Most of
the algorithms above require a sample from the model (\ref{21}). This
is most often done by using a local Markov chain based on adding or
deleting edges via Metropolis or Glauber dynamics (Gibbs sampling).
These authors
show that if the parameters are nonnegative, then for large $n$,
\begin{itemize}
\item either the $p_\beta$ model is essentially the same as an Erd\H
{o}s--R\'enyi
model (in which case the Markov chain mixes in $n^2\log n$
steps);
\item or the Markov chain takes exponential time to mix.
\end{itemize}
Thus, in cases where the model is not essentially trivial, the Markov
chains required to carry MCMLE procedures cannot be usefully run to
stationarity.

Two other approaches to estimation are worth mentioning. The
pseudo-likelihood approach of Besag \cite{besag75} is widely used
because of its
ease of implementation. Its properties are at best poorly understood:
it does not directly \mbox{maximize} the likelihood and in empirical
comparisions (see, e.g., \cite{corander}), has appreciably larger
variability than the MLE. Comets and Jan\v{z}ura \cite{cj98} prove
consistency and asymptotic normality of the maximum pseudo-likelihood
estimator in certain Markov random field models. Chatterjee \cite
{chatterjee07} shows that it is consistent for estimating the
temperature parameter of the Sherrington--Kirkpatrick model of spin
glasses. The second approach
is Snijders' \cite{snijders09} suggestion to use the Robbins--Monro
optimization procedure \cite{rm51} to compute solutions to the moment equations
$E_\beta(T(G))=T(G^*)$ where $G^*$ is the observed graph. While
promising, the approach requires generating points from $p_\beta$ for
arbitrary $\beta$. The only way to do this at present is by MCMC and
the results of \cite{bhamidi} suggest this may be impractical.

\subsection{Graph limits}\label{sec2c}

In a sequence of papers \cite
{borgsetal06,borgsetal08,borgsetal07,freedmanlovaszschrijver07,lovasz06,lovasz07,lovaszsos08,lovaszszegedy06,lovaszszegedy07,lovaszszegedy07b,lovaszszegedy09},
Laszlo Lov\'asz and coauthors V. T. S\'os, B.
Szegedy, C. Borgs, J. Chayes, K. Vesztergombi, A.~Schrijver, and
M. Freedman have developed a beautiful, unifying theory of graph
limits. (See also the related work of Austin \cite{austin08} and
Diaconis and Janson \cite{pd175} which traces this back to work of
Aldous \cite{aldous81}, Hoover \cite{hoover82} and Kallenberg \cite
{kallenberg05}.)
This body of work sheds light on various graph-theoretic topics such as
graph homomorphisms, Szemer\'edi's
regularity lemma, quasi-random graphs, graph testing and extremal
graph theory, and has even found applications in statistics and
related areas (see, e.g., \cite{pd177}). Their theory has been
developed for dense graphs (number of edges comparable to the square
of number of vertices) but parallel theories for sparse graphs are
beginning to emerge \cite{bollobasriordan09}. 

Lov\'asz and coauthors define the limit of a sequence of dense graphs
as follows. We quote the definition verbatim from
\cite{lovaszszegedy06} (see also
\cite{borgsetal08,borgsetal07,pd175}). Let $G_n$ be a sequence of
simple graphs whose number of nodes tends to infinity. For every fixed
simple graph $H$, let $|{\hom(H, G)}|$ denote the number of
homomorphisms of $H$ into $G$ [i.e., edge-preserving maps $V(H) \to
V(G)$, where $V(H)$ and $V(G)$ are the vertex sets]. This number is
normalized to get the homomorphism density
%
%
\begin{equation}
\label{homdens}
t(H,G):= \frac{|{\hom(H, G)}|}{|V(G)|^{|V(H)|}}.
\end{equation}
This gives the probability that a random mapping $V(H) \to V(G)$ is a
homomorphism.

Note that $|{\hom(H,G)}|$ is not the count of the number of copies of
$H$ in $G$, but is a constant multiple of that if $H$ is a complete
graph. For example, if $H$ is a
triangle, $|{\hom(H,G)}|$ is the number of triangles in $G$ multiplied
by six. On the other hand if $H$ is, say, a $2$-star (i.e., a triangle
with one edge missing) and $G$ is a triangle, then the number of copies
of $H$ in $G$ is zero, while $|{\hom(H, G)}| = 3^3 = 27$.

Suppose that the graphs $G_n$ become more and more similar in the
sense that $t(H, G_n)$ tends to a limit $t(H)$ for every $H$. One way
to define a limit of the sequence $\{G_n\}$ is to define an
appropriate limit object from which the values $t(H)$ can be read off.

The main result of \cite{lovaszszegedy06} (following the earlier
equivalent work of Aldous \cite{aldous81} and Hoover \cite{hoover82})
is that indeed there is a natural ``limit object'' in the form of a
function $h\in\mm$, where $\mm$ is the space of all measurable
functions from $[0,1]^2$ into $[0,1]$ that satisfy $h(x,y)=h(y,x)$ for
all $x,y$.

Conversely, every such function arises as the limit of an appropriate
graph sequence. This limit object determines all the limits of
subgraph densities: if $H$ is a simple graph with $V(H) = [k] = \{
1,\ldots, k\}$, let
%
%
\begin{equation}
\label{tfdef} t(H,h) = \int_{[0,1]^k}\prod
_{\{i,j\}\in E(H)} h(x_i, x_j) \,dx_1
\cdots dx_k.
\end{equation}
Here $E(H)$ denotes the edge set of $H$. A sequence of graphs
$\{G_n\}_{n\ge1}$ is said to converge to $h$ if for every finite
simple graph $H$,
%
%
\begin{equation}
\label{gconv} \lim_{n\to\infty} t(H, G_n) = t(H,h).
\end{equation}
Intuitively, the interval $[0,1]$ represents a ``continuum'' of
vertices, and $h(x,y)$ denotes the probability of putting an edge
between $x$ and $y$. For example, for the Erd\H{o}s--R\'enyi graph
$G(n,p)$, if $p$ is fixed and $n \to\infty$, then the limit graph is
represented by the function that is identically equal to $p$
on $[0,1]^2$. Clearly, this framework is therefore useful only when
$p$ does not tend to zero when $n\ra\infty$, that is, in the case of
dense Erd\H{o}s--R\'enyi graphs.

These limit objects, that is, elements of $\mm$, are called ``graph
limits'' or ``graphons'' in \cite
{lovaszszegedy06,borgsetal08,borgsetal07}. A finite
simple graph $G$ on $\{1,\ldots,n\}$ can also be represented as a
graph limit $f^G$ is a natural way, by defining
%
%
\begin{equation}
\label{wg} f^G(x,y) = \cases{ 1, &\quad if $\bigl(\lceil nx\rceil,
\lceil
ny\rceil\bigr)$ is an edge in $G$,
\cr
0, &\quad otherwise.}
\end{equation}
The definition makes sense because $t(H, f^G) = t(H, G)$ for every
simple graph $H$ and therefore the constant sequence $\{G, G,\ldots\}$
converges to the graph limit~$f^G$. Note that this allows \textit{all}
simple graphs, irrespective of the
number of vertices, to be represented as elements of a single abstract
space, namely $\mm$.

With the above representation, it turns out that the notion of
convergence in terms of subgraph densities outlined above can be
captured by an explicit metric on $\mm$, the so-called \textit{cut
distance} (originally defined for finite graphs by Frieze and Kannan
\cite{friezekannan99}). Start with the space $\mm$ of measurable
functions $f(x,y)$ on $[0,1]^2$ that satisfy $0\le f(x,y)\le1$ and
$f(x,y)=f(y,x)$. Define the cut distance
%
%
\begin{equation}
\label{defcut} d_\square(f,g):= \sup_{S,T\subseteq[0,1]} \biggl|\int
_{S\times T} \bigl[f(x,y)-g(x,y)\bigr] \,dx \,dy \biggr|.
\end{equation}
Introduce in $\mm$ an equivalence relation: let $\Sigma$ be the space
of measure preserving bijections $\sigma\dvtx[0,1]\to[0,1]$. Say that
$f(x,y)\sim g(x,y)$ if $f(x,y)=g_\sigma(x,y):=g(\sigma x, \sigma y)$
for some $\sigma\in\Sigma$. Denote by ${\widetilde g}$ the closure in
$(\mm, d_\Box)$ of the orbit $\{g_\sigma\}$. The quotient space is
denoted by $\mmm$ and $\tau$ denotes the natural map $g\to
{\widetilde
g}$. Since $d_\Box$ is invariant under $\sigma$ one can define on
$\mmm$, the natural distance $\delta_\Box$ by
\[
\delta_\Box({\widetilde f},{\widetilde g}):=\inf_\sigma
d_\Box(f, g_\sigma)=\inf_\sigma
d_\Box(f_\sigma, g)=\inf_{\sigma_1,\sigma_2}d_\Box(f_{\sigma_1},
g_{\sigma_2})
\]
making $(\mmm, \delta_\Box)$ into a metric space. To any\vspace*{2pt} finite graph
$G$, we associate $f^G$ as in (\ref{wg}) and its orbit ${\widetilde
G}=\tau f^G= {\widetilde f}^G\in\mmm$.

The papers by Lov\'asz and coauthors establish many important
properties of the metric space $\mmm$ and the associated notion of
graph limits. For example, $\mmm$ is compact. A pressing objective is
to understand what functions from $\mmm$ into $\rr$ are continuous.
Fortunately, it is an easy fact that the homomorphism density $t(H,
\cdot)$ is continuous for any finite simple graph $H$
\cite{borgsetal07,borgsetal08}. There are other, more complicated
functions that are continuous. For example, the degree distribution is
continuous with respect to this topology, as is the distribution of
eigenvalues. See \cite{austin10,austin08} for further discussions.

\subsection{Large deviations for random graphs}\label{sec2d}

Let $G(n,p)$ be the random graph on $n$ vertices where each edge is
added independently with probability $p$. This model has been the
subject of extensive investigations since the pioneering work of
Erd\H{o}s and R\'enyi \cite{erdosrenyi60}, yielding a large body of
literature (see \cite{bollobas01,JLR00} for partial surveys).

Recently, Chatterjee and Varadhan \cite{cv10} formulated a large
deviation principle for the Erd\H{o}s--R\'enyi graph, in the same way
as Sanov's
theorem \cite{sanov61} gives a large deviation principle for an
i.i.d. sample. The formulation and proof of this result makes
extensive use of the properties of the topology described in
Section \ref{sec2c}.\eject

Let $I_p\dvtx[0,1] \to\rr$ be the function
%
%
\begin{equation}
\label{ipdef1} I_p(u):= \frac{1}{2}u\log\frac{u}{p}+
\frac{1}{2}(1-u)\log\frac
{1-u} {1-p}.
\end{equation}
The domain of the function $I_p$ can be extended to $\mm$ as
%
%
\begin{equation}\label{ipdef2}
I_p(h):=\int_0^1\int
_0^1 I_p\bigl(h(x,y)\bigr)\,dx\,dy.
\end{equation}
The function $I_p$ can be defined on $\mmm$ by declaring
$I_p(\widetilde{h}):= I_p(h)$ where $h$ is any representative element
of the equivalence class $\widetilde{h}$. Of course, this raises the
question whether $I_p$ is well defined on $\mmm$. It was proved in
\cite{cv10} that the function $I_p$ is indeed well defined on $\mmm$
and is lower semicontinuous under the cut metric~$\delta_\Box$.

The random graph $G(n,p)$ induces probability distributions
$\pp_{n,p}$ on the space\vspace*{1pt} $\mm$ through the map $G\to f^G$ and
${\widetilde\pp}_{n,p}$ on $\mmm$ through the\vspace*{2pt} map $G\to f^G\to
{\widetilde f}^G= \widetilde{G}$. The large deviation principle for
${\widetilde\pp}_{n,p}$ on $(\mmm, \delta_\Box)$ is the main
result of
\cite{cv10}.
%
%
\begin{theorem}[(Chatterjee and Varadhan \cite{cv10})]\label{cvmain}
For each fixed $p\in(0,1)$, the sequence ${\widetilde\pp}_{n,p}$
obeys a large deviation principle in the space $\mmm$ (equipped with
the cut metric) with rate function $I_p$ defined by (\ref{ipdef2}).
Explicitly, this means that for any closed set $\widetilde{F}
\subseteq\mmm$,
%
%
\begin{equation}
\label{closed} \limsup_{n\to\infty} \frac{1}{n^2}\log{\widetilde
\pp}_{n,p}(\widetilde{F}) \le-\inf_{{\widetilde h}\in\widetilde{F}}
I_p({\widetilde h})
\end{equation}
and for any open set $\widetilde{U}\subseteq\mmm$,
%
%
\begin{equation}
\label{open} \liminf_{n\to\infty} \frac{1}{n^2}\log{\widetilde
\pp}_{n,p}( \widetilde{U}) \ge-\inf_{{\widetilde h}\in\widetilde{U}}
I_p({\widetilde h}).
\end{equation}
\end{theorem}

\section{The main result}\label{sec3}

Let $T\dvtx\mmm\to\rr$ be a bounded continuous function on the metric
space $(\mmm,\delta_\Box)$. Fix $n$ and let $\mg_n$ denote the set of
simple graphs on $n$ vertices. Then $T$ induces a probability mass
function $p_n$ on $\mg_n$ defined as
\[
p_n(G):=e^{n^2 (T(\widetilde{G})-\psi_n)}.
\]
Here $\widetilde{G}$ is the image of $G$ in the quotient space $\mmm$ as
defined in Section \ref{sec2c} and $\psi_n$ is a constant such that
the total
mass of $p_n$ is $1$. Explicitly,
%
%
\begin{equation}
\label{psidef} \psi_n=\frac1{n^2}\log\sum
_{G\in\mg_n}e^{n^2 T(\widetilde{G})}.
\end{equation}
The coefficient $n^2$ is meant to ensure that $\psi_n$ tends to a
nontrivial limit as \mbox{$n\to\infty$}. (Note that $T$ does not vary with
$n$.)\vadjust{\goodbreak} To describe this limit, define a
function $I\dvtx[0,1]\to\rr$ as
\[
I(u):=\tfrac12 u\log u+\tfrac12(1-u)\log(1-u)
\]
and extend $I$ to $\mmm$ in the usual manner:
%
%
\begin{equation}
\label{idef} I(\widetilde{h})=\iint_{[0,1]^2}I\bigl(h(x,y)\bigr) \,dx\,dy,
\end{equation}
where $h$ is a representative element of the equivalence class
$\widetilde{h}$. As mentioned before, it follows from a result of
\cite{cv10} that $I$ is well defined and lower semi-continuous on
$\mmm$. The following theorem is the first main result of this paper.
%
%
\begin{theorem}\label{soln}
If $T\dvtx\mmm\to\rr$ is a bounded continuous function and $\psi_n$ and
$I$ are defined as above, then
\[
\psi:=\lim_{n\to\infty} \psi_n = \sup
_{\widetilde{h}\in\mmm} \bigl(T(\widetilde{h}) - I(\widetilde{h})\bigr).
\]
\end{theorem}
We will see later that the supremum on the right-hand side is actually
a maximum, that is, there is some $\widetilde{h}$ where the supremum
is attained. This is significant because such maximizing $\widetilde
{h}$'s describe the structure of the random graph in the large $n$ limit.

As mentioned in the \hyperref[sec1]{Introduction}, evaluation of the
normalizing constant is one of the key problems in statistical
applications of exponential random graphs. Incidentally, even the
existence of the limit in Theorem \ref{soln} has an important
consequence. Suppose that a computer program can evaluate the exact
value of the normalizing constant for moderate sized $n$. Then if $n$
is large, one can choose a ``scaled down'' model with a smaller number
of nodes, and use the exact value of the normalizing constant in the
scaled down model as an approximation to the normalizing constant in
the larger model.

Theorem \ref{soln} gives an asymptotic formula for $\psi_n$. However,
it says nothing about the behavior of a random graph drawn from the
exponential random graph model. Some aspects of this behavior can be
described as follows. Let $\tilf^*$ be the subset of $\mmm$ where
$T(\tilh)-I(\tilh)$ is maximized. By the compactness of $\mmm$, the
continuity of $T$ and the lower semi-continuity of $I$, $\tilf^*$ is a
nonempty compact set. Let $G_n$ be a random graph on $n$ vertices
drawn from the exponential random graph model defined by $T$. The
following theorem shows that for $n$ large, $\widetilde{G}_n$ must lie
close to $\tilf^*$ with high probability. In particular, if $\tilf^*$
is a
singleton set, then the theorem gives a weak law of large numbers for
$G_n$.
%
%
\begin{theorem}\label{limitbehave}
Let $\tilf^*$ and $G_n$ be defined as in the above paragraph. Let $\pp
$ denote the probability measure on the underlying probability space on
which $G_n$ is defined. Then for\vadjust{\goodbreak}
any $\eta> 0$ there exist $C, \gamma>0$ such that for any $n$,
\[
\pp\bigl(\delta_\Box\bigl(\widetilde{G}_n, \tilf^*\bigr)
> \eta\bigr) \le Ce^{-n^2
\gamma}.
\]
\end{theorem}

\section{An application}\label{sec4}

Let $H_1,\ldots, H_k$ be finite simple graphs, where $H_1$ is the
complete graph on two vertices (i.e., just a single edge), and each
$H_i$ contains at least one edge. Let $\beta_1,\ldots,\beta_k$ be $k$
real numbers. For any $h\in\mm$, let
%
%
\begin{equation}
\label{tdef1} T(h):= \sum_{i=1}^k
\beta_i t(H_i, h),
\end{equation}
where $t(H_i, h)$ is the homomorphism density of $H_i$ in $h$, defined
in (\ref{tfdef}). Note that there is nothing special about taking
$H_1$ to be a single edge; if we do not want $H_1$ in our sufficient
statistic, we just take $\beta_1=0$; all theorems would remain valid.

As remarked in Section \ref{sec2c}, $T$ is
continuous with respect to the cut distance on~$\mm$, and hence admits
a natural definition on $\mmm$. Note that for any finite simple graph
$G$ that has at least as many nodes as the largest of the~$H_i$'s,
\[
T(\widetilde{G}) = \sum_{i=1}^k
\beta_i t(H_i, G).
\]
For example, if $k=2$, and $H_2$ is a triangle, and $G$ has at least
$3$ nodes, then
\[
T(\widetilde{G}) = \frac{2\beta_1 (\#\mbox{edges in } G)}{n^2} + \frac
{6\beta_2 (\#\mbox{triangles in } G)}{n^3}.
\]
Let $\psi_n$ be as in (\ref{psidef}), and let $G_n$ be the $n$-vertex
exponential random graph with sufficient statistic $T$. Theorem \ref{soln}
gives a formula for $\lim_{n\to\infty} \psi_n$ as the
solution of a variational problem. Surprisingly the variational
problem is explicitly solvable if $\beta_2,\ldots,\beta_k$ are
nonnegative.
%
%
\begin{theorem}\label{specialnorm}
Let $T$, $\psi_n$ and $H_1,\ldots, H_k$ be as above. Suppose
$\beta_2,\ldots,\beta_k$ are nonnegative. Then
%
%
\begin{equation}
\label{scalar} \lim_{n\to\infty} \psi_n = \sup
_{0\le u\le1} \Biggl(\sum_{i=1}^k
\beta_i u^{e(H_i)} - I(u) \Biggr),
\end{equation}
where $I(u)=\frac12 u\log u+\frac12(1-u)\log(1-u)$ and $e(H_i)$ is the
number of edges in~$H_i$. Moreover, each solution of the variational
problem of Theorem \ref{soln} for this $T$ is a constant function,
where the constant solves the scalar maximization
problem~(\ref{scalar}).
\end{theorem}
Theorem \ref{specialnorm} gives the limiting value of $\psi_n$ if
$\beta_2,\ldots,\beta_k$ are nonnegative. The next theorem describes
the behavior of the exponential random graph $G_n$ under this
condition if $n$ is large.\eject
%
%
\begin{theorem}\label{specialbehave}
For each $n$, let $G_n$ be an $n$-vertex exponential random graph
with sufficient statistic $T$ defined in (\ref{tdef1}). Assume that
$\beta_2,\ldots,\beta_k$ are nonnegative. Then:
\begin{longlist}[(a)]
\item[(a)] If the maximization problem in (\ref{scalar}) is solved at a
unique value $u^*$, then $G_n$ is indistinguishable from the Erd\H
{o}s--R\'enyi
graph $G(n,u^*)$ in the large $n$ limit, in the sense that
$\widetilde{G}_n$ converges to the constant function $u^*$ in
probability as $n\to\infty$.
\item[(b)] Even if the maximizer is not unique, the set $U$ of maximizers
is a finite subset of $[0,1]$ and
\[
\min_{u\in U}\delta_{\Box}(\widetilde{G}_n,
\widetilde{u}) \to0 \qquad\mbox{in probability as $n\to\infty$},
\]
where $\widetilde{u}$ denotes the image of the constant function $u$
in $\mmm$. In other words, $G_n$ behaves like an Erd\H{o}s--R\'enyi
graph $G(n,u)$ where $u$ is picked randomly from some probability
distribution on $U$.
\end{longlist}
\end{theorem}
It may be noted here that the conclusion of Theorem \ref{specialbehave} was
proved earlier by Bhamidi et al. \cite{bhamidi} under certain
restrictions on the parameters that they called a ``high temperature
condition.'' This is in analogy with spin systems, since random graphs
may be interpreted as systems of particles (corresponding to edges)
each having spin 0 or 1 (i.e., closed or open). With this
interpretation, it is straightforward to check that when $\beta
_2,\ldots, \beta_k$ are nonnegative, the model defined above
satisfies the so-called FKG property \cite{fkg71}. Stated simply, the
FKG property means that if $f$ and $g$ are monotone functions of the
random graph (i.e., functions whose values cannot decrease if more
edges are added to the graph), then $f$ and $g$ are positively
correlated random variables. The FKG property has important
consequences; for instance, it implies that the expected value of
$t(H_i, G)$ is an increasing function of $\beta_j$ for any $i$ and
$j$. We will see some further consequences of the FKG property in our
proof of Theorem \ref{degen} in the next section.

\section{Phase transitions and near-degeneracy}\label{sec5}
To illustrate the results of the previous section, recall the
exponential random graph model (\ref{trianglemodel}) with edges and
triangles as sufficient statistics:
%
%
\begin{eqnarray}
\label{ttriangle} T(\widetilde{G}) &=& 2\beta_1 \frac{\#\mathrm{edges}\
\mathrm{in}\
G}{n^2} + 6
\beta_2 \frac{\#\mathrm{triangles}\ \mathrm{in}\ G}{n^3}
\nonumber\\[-8pt]\\[-8pt]
&=& \beta_1 t(H_1, G) + \beta_2
t(H_2, G),
\nonumber
\end{eqnarray}
where $H_1$ is a single edge and $H_2$ is a triangle. Fix $\beta_1$
and $\beta_2$ and let
%
%
\begin{equation}
\label{lu} \ell(u):= \beta_1 u + \beta_2 u^3
- I(u),
\end{equation}
where $I(u)=\frac12 u\log u+\frac12(1-u)\log(1-u)$, as usual. Let $U$
be the set of maximizers of $\ell(u)$ in $[0,1]$. Theorem \ref{specialbehave}
describes the limiting behavior of $G_n$ in terms
of the set $U$. In particular, if $U$ consists of a single point $u^*
= u^*(\beta_1,\beta_2)$, then $G_n$ behaves like the Erd\H{o}s--R\'enyi graph
$G(n,u^*)$ when $n$ is large.

%
%
\begin{figure}

\includegraphics{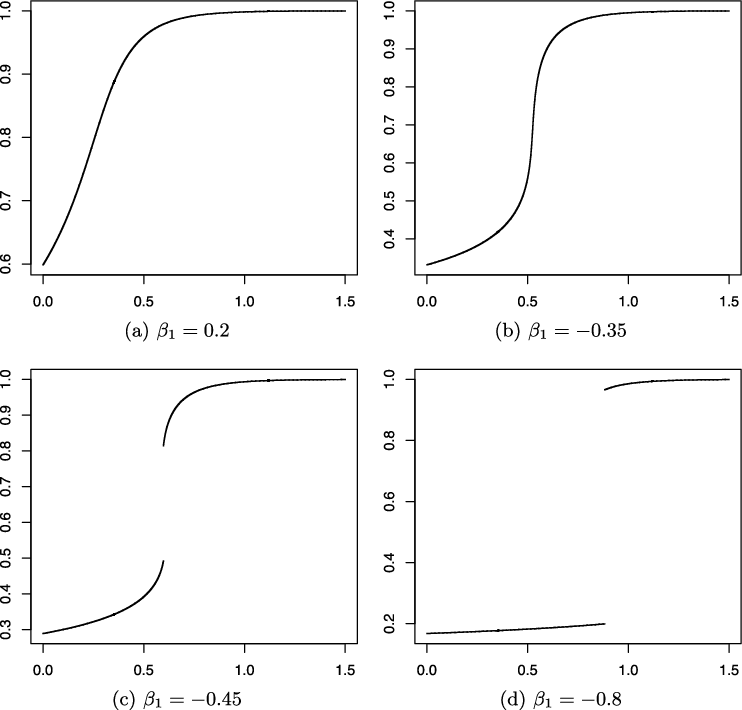}

\caption{Plot of $u^*(\beta_1,\beta_2)$ on $y$-axis vs $\beta_2$ on
$x$-axis for different fixed values of $\beta_1$. Part \textup{(c)}
demonstrates a phase transition. Part \textup{(d)} demonstrates
near-degeneracy.}
\label{pstarfig}
\end{figure}

It is likely that $u^*(\beta_1,\beta_2)$ does not have a closed form
expression, other than when $\beta_2 = 0$, in which case
\[
u^*(\beta_1,0) = \frac{e^{2\beta_1}}{1+e^{2\beta_1}}.
\]
It is, however, quite easy to numerically approximate
$u^*(\beta_1,\beta_2)$. Figure~\ref{pstarfig} plots $u^*(\beta_1,\beta_2)$
versus $\beta_2$ for four different fixed values of $\beta_1$, namely,
$\beta_1=0.2$, $-0.35,-0.45$, and $-0.8$. The figures show that $u^*$ is
a continuous function of $\beta_2$ as long as $\beta_1$ is not too far
down the negative axis.

But for $\beta_1$ below a threshold (e.g., when $\beta_1=-0.45$),
$u^*$ shows a single jump discontinuity in $\beta_2$, signifying a
phase transition. In physical terms, this is a first order phase
transition, by the following logic. By Theorem~\ref{specialbehave}, our random
graph behaves like $G(n, u^*)$ when $n$ is large. On the other hand, by
a standard computation the expect number of triangles is the first
derivative of the free energy $\psi_n$ with respect to $\beta_2$.
Therefore in the large $n$ limit, a discontinuity in $u^*$ as a
function of $\beta_2$ signifies a discontinuity in the derivative of
the limiting free energy, which is the physical definition of a first
order phase transition.

At the point of discontinuity, $\ell(u)$ is
maximized at two values of $u$, that is, the set $U$ consists of two
points. Lastly, as $\beta_1$ goes down the negative axis, the model
starts to exhibit ``near-degeneracy'' in the sense of Handcock
\cite{handcock} (see also \cite{park04}) as seen in the
last frame of Figure \ref{pstarfig}. This means that when $\beta_1$ is a
large negative number, then as $\beta_2$ varies, the
model transitions from being a very sparse graph for low values of
$\beta_2$, to a very dense graph for large values of
$\beta_2$, completely skipping all intermediate structures. If this
sentence is confusing, please see Theorem \ref{degen} below for a precise
statement. This theorem gives a simple mathematical description of this
phenomenon and hence the first rigorous proof of the degeneracy
observed in exponential graph models. Related results are in H\"
aggstrom and Jonasson \cite{hj99}.

%
%
\begin{theorem}\label{degen}
Let $G_n$ be an exponential random graph with sufficient statistic $T$
defined in (\ref{ttriangle}) and let $\pp$ be the probability measure
on the underlying probability space on which $G_n$ is defined. Fix any
$\beta_1 < 0$. Let
\[
c_1:= \frac{e^{\beta_1}}{1+e^{\beta_1}},\qquad c_2:= 1+
\frac
{1}{2\beta_1}.
\]
Suppose $|\beta_1|$ is so large that $c_1 < c_2$. Let $e(G_n)$ be the
number of edges in $G_n$ and let $f(G_n):= e(G_n)/{n\choose2}$ be the
edge density. [Note that $f(G_n) = \frac{n}{n-1} t(H_1, G_n)$, where
$H_1$ is a single edge.]

Then there exists $q = q(\beta_1) \in[0,\infty)$ such that if
$-\infty< \beta_2 < q$, then
\[
\lim_{n\ra\infty} \pp\bigl(f(G_n)> c_1\bigr)
=0,
\]
and if $\beta_2 > q$, then
\[
\lim_{n\ra\infty} \pp\bigl(f(G_n)< c_2\bigr)
=0.
\]
In other words, if $\beta_1$ is a large negative number, then $G_n$ is
either sparse (if $\beta_2 < q$) or nearly complete (if $\beta_2> q$).
\end{theorem}

The difference in the values of $c_1$ and $c_2$ can be quite striking
even for relatively small values of $\beta_1$.\vadjust{\goodbreak} For example, $\beta_1
=-5$ gives $c_1 \simeq0.007$ and $c_2 = 0.9$. Significant extensions
of Theorem \ref{degen} have been made in the recent manuscripts \cite
{aristoffradin,radin1,radin2,yin}.

\section{The symmetric phase, symmetry breaking, and the
Euler--Lagrange equations}\label{sec6}
The purpose of this section is to extend the analysis of the model from
Section \ref{sec4} beyond the case of nonnegative parameters. We begin
with a standard approach to solving variational problems.

\subsection{Euler--Lagrange equations}\label{sec62}

We return to the exponential random graph model with sufficient
statistic $T$ defined in (\ref{tdef1}) in terms of the densities of
$k$ fixed graphs $H_1,\ldots, H_k$, where $H_1$ is a single edge.
Theorems \ref{specialnorm} and \ref{specialbehave} analyze this model
when $\beta_2,\ldots,\beta_k$ are nonnegative. What if they are not?
One can still try to derive the Euler--Lagrange equations (or\vspace*{1pt} Euler's
equation; see \cite{gelfand}) for the
related variational problem of maximizing $T(\tilh)-I(\tilh)$. The
following theorem presents the outcome of this effort.

For a finite simple graph $H$, let $V(H)$ and $E(H)$ denote the sets
of vertices and edges of $H$. Given a symmetric measurable function
$h\dvtx[0,1]^2\to\rr$, for each $(r,s)\in E(H)$ and each pair of points
$x_r, x_s\in[0,1]$, define
\[
\Delta_{H,r,s}h(x_r,x_s):=\int
_{[0,1]^{|V(H)\setminus\{r,s\}|}} \mathop{\prod_{(r',s')\in
E(H)}}_{(r',s')\neq(r,s)}h(x_{r'},x_{s'})
\mathop{\prod_{v\in V(H)}}_{v\neq r,s}
\,dx_v.
\]
For $x,y\in[0,1]$ define
%
%
\begin{equation}
\label{deltahi} \Delta_{H} h(x,y):= \sum
_{(r,s)\in E(H)} \Delta_{H, r,s}h (x,y).
\end{equation}
For example, when $H$ is a triangle, then $V(H)=\{1,2,3\}$ and
\[
\Delta_{H, 1,2}h(x,y) = \Delta_{H, 1,3}h(x,y) =
\Delta_{H,
2,3}h(x,y) = \int_0^1 h(x,z)
h(y,z)\,dz
\]
and therefore $\Delta_{H} h(x,y) = 3\int_0^1 h(x,z) h(y,z) \,dz$. When
$H$ contains exactly one edge, define $\Delta_H h \equiv1$ for any
$h$, by the usual convention that the empty product is~$1$. The
following theorem gives the Euler--Lagrange equations for the optimizer
$h$ of Theorem \ref{soln} in terms of these $\Delta_Hh$'s.
%
%
\begin{theorem}\label{euler}
Let\vspace*{1pt} $T\dvtx\mmm\to\rr$ be defined as in (\ref{tdef1}) and the operator
$\Delta_H$ be defined as in (\ref{deltahi}). If $\tilh\in\mmm$
maximizes $T(\tilh)-I(\tilh)$, then any representative element $h\in
\tilh$ must satisfy for almost all $(x,y)\in[0,1]^2$,
\[
h(x,y) = \frac{e^{2\sum_{i=1}^k \beta_i \Delta_{H_i} h(x,y)}}{1 +
e^{2\sum_{i=1}^k \beta_i \Delta_{H_i} h(x,y)}}.
\]
Moreover, any maximizing function must be bounded away from $0$ and
$1$.
\end{theorem}\eject
Unfortunately, these equations may have many solutions and therefore do
not uniquely identify the optimizer. The next subsection gives a
sufficient condition under which the solution in unique.

\subsection{The replica symmetric phase}
Borrowing terminology from spin\break glasses, we define the \textit{replica
symmetric phase} or simply the \textit{symmetric phase} of a
variational problem like maximizing $T(h)-I(h)$ as the set of
parameter values for which all the maximizers are constant functions.
When the parameters are such that all maximizers are nonconstant
functions we say that the parameter vector is in the region of broken
replica symmetry, or simply broken symmetry. There may be another
situation, where some optimizers are constant while others are
nonconstant, although we do not know of such examples. (This third
region may be called a region of partial symmetry.)

Statistically, the exponential random graph behaves like an Erd\H
{o}s--R\'enyi graph
in the symmetric region of the parameter space, while such behavior
breaks down in the region of broken symmetry. This follows easily from
Theorem \ref{limitbehave}.

Theorem \ref{specialbehave} shows that for the sufficient statistic
$T$ defined in (\ref{tdef1}), each
$(\beta_1,\beta_2,\ldots,\beta_k)$ in $ \rr\times\rr_+^{k-1}$
falls in
the replica symmetric region. Does symmetry hold only when
$\beta_2,\ldots,\beta_k$ are nonnegative? The following theorem
(proven with the aid of the Euler--Lagrange equations of Theorem \ref{euler}),
shows that this is not the case;
$(\beta_1,\ldots,\beta_k)$ is in the replica symmetric region whenever
$|\beta_2|,\ldots,|\beta_k|$ are small enough. Of course, this does not
supersede Theorem \ref{specialbehave} since it does not cover large
positive values of $\beta_2,\ldots,\beta_k$. However, it proves replica
symmetry for small negative values of $\beta_2,\ldots,\beta_k$, which
is not covered by Theorem \ref{specialbehave}.
%
%
\begin{theorem}\label{contraction}
Consider the exponential random graph with sufficient statistic $T$
defined in (\ref{tdef1}). Suppose $\beta_1,\ldots, \beta_k$ are such
that
\[
\sum_{i=2}^k|\beta_i|
e(H_i) \bigl(e(H_i)-1\bigr) < 2,
\]
where $e(H_i)$ is the number of edges in $H_i$. Then the conclusions
of Theorems \ref{specialnorm} and \ref{specialbehave} hold true for
this value of the parameter vector $(\beta_1,\ldots,\beta_k)$.
\end{theorem}

\subsection{Symmetry breaking}\label{sec61}
Theorems \ref{specialbehave} and \ref{contraction} establish various
regions of symmetry in the exponential random graph model with
sufficient statistic $T$ defined in (\ref{tdef1}). That leaves the
question: is there a region where symmetry breaks? We specialize to
the simple case where $k=2$ and $H_2$ is a triangle, that is, the example
of Section \ref{sec5}. In this case, it turns out that replica
symmetry breaks
whenever $\beta_2$ is less than a sufficiently large negative number
depending on $\beta_1$.
%
%
\begin{theorem}\label{symmbreaking}
Consider the exponential random graph with sufficient statistic $T$
defined in (\ref{ttriangle}). Then for any given value of $\beta_1$,
there is a positive constant $C(\beta_1)$ sufficiently large so that
whenever $\beta_2 < -C(\beta_1)$, $T(h)-I(h)$ is not maximized at
any constant function. Consequently, if $G_n$ is an $n$-vertex
exponential random graph with this sufficient statistic, then there
exists $\ep> 0$ such that
\[
\lim_{n\to\infty}\pp\bigl(\delta_\Box(\widetilde
{G}_n,\widetilde{C})>\ep\bigr)=1,
\]
where $\widetilde{C}$ is the set of constant functions. In other
words, $G_n$ does not look like an Erd\H{o}s--R\'enyi graph in the
large $n$ limit.
\end{theorem}
For interesting recent developments about symmetry breaking in
exponential random graph models, see Lubetzky and Zhao \cite{lz}.

\subsection{A completely solvable case}
A $j$-star is an undirected graph with one ``root'' vertex and $j$ other
vertices connected to the root vertex, with no edges between any of
these $j$ vertices. Let $H_j$ be a $j$-star for $j=1,\ldots, k$. Let
$T$ be the sufficient statistic
%
%
\begin{equation}
\label{pstarstat} T(G) = \sum_{j=1}^k
\beta_j t(H_j, G).
\end{equation}
%
Theorems \ref{specialnorm} and \ref{specialbehave} describe the
behavior of
this model when $\beta_2,\ldots, \beta_k$ are all nonnegative. The
following theorem completely solves this model for all values of $\beta
_2,\ldots, \beta_k$. The proof of this theorem was suggested by the
anonymous referee, improving upon the version of the result given in an
earlier draft. 
%
%
\begin{theorem}\label{pstarcase}
For the sufficient statistic $T$ defined in (\ref{pstarstat}), the
conclusions of Theorems \ref{specialnorm} and \ref{specialbehave} hold
for any
$\beta_1,\ldots, \beta_k\in\rr$.
\end{theorem}


\section{Extremal behavior}\label{sec7}
In the sections above, we have been assuming that
$\beta_2,\ldots,\beta_k$ are positive or barely negative. In this
section, we investigate what happens when $k=2$ and $\beta_2$ is large
and negative. The limits are describable but far from
Erd\H{o}s--R\'enyi. Our work here is inspired by related results of
Sukhada Fadvanis who has a different argument (using Tur\'an's theorem
\cite{turan}) for the case of triangles.

Suppose $H$ is any finite simple graph containing at least one edge.
Let $T$ be the sufficient statistic
\[
T(\widetilde{G}) = 2\beta_1\frac{\#\mathrm{edges}\ \mathrm{in}\
G}{n^2} + \beta_{2}
t(H, G).
\]
%
Let $G_n$ be the exponential random graph on $n$ vertices with this
sufficient statistic and let $\psi_n$ be the associated normalizing
constant as defined in (\ref{psidef}). Then Theorem \ref{soln} gives
\[
\lim_{n\ra\infty} \psi_n = \sup_{h\in\mm}
\bigl(T(h)-I(h)\bigr) =:\psi,
\]
where $I$ is defined in (\ref{idef}).
We also know (by Theorem \ref{limitbehave}) that
\[
\delta_\Box\bigl(\widetilde{G}_n, \widetilde{F}^*\bigr)
\ra0 \qquad\mbox{in probability as $n\to\infty$,}
\]
where $\widetilde{F}^*$ is the subset of $\mmm$ where $T-I$ is
maximized. (Note that $\WF^*$ is a closed set since $T-I$ is an upper
semicontinuous map.)

We can compute $\WF^*$ and $\psi$ when $\beta_2$ is positive, or
negative with small magnitude. We are unable to carry out the explicit
computation in the case of large negative $\beta_2$, unless $H$ is a
convenient object like a $j$-star. However, a qualitative description
can still be given by analyzing the behavior of $\WF^*$ and $\psi$ as
$\beta_2\ra-\infty$. Fixing $\beta_1$, we consider these objects as
functions of $\beta_2$ and write $\WF^*(\beta_2)$, $\psi(\beta_2)$
and $T_{\beta_2}$ instead of $\WF^*$, $\psi$ and $T$. Recall that
the chromatic number of a graph is the minimum number of colors
required to color the edges so that no two neighbors get the same color.
%
%
\begin{theorem}\label{erdos}
Fixing $H$ and $\beta_1$, let $\WF^*(\beta_2)$ and $\psi(\beta_2)$
be as above. Let $\chi(H)$ be the chromatic number of $H$, and define
%
%
\begin{equation}
\label{chrom} g(x,y):= \cases{ 1, &\quad if $\bigl[\bigl(\chi(H)-1\bigr
)x\bigr] \ne
\bigl[\bigl(\chi(H)-1\bigr)y\bigr]$,
\cr
0, &\quad otherwise,}
\end{equation}
where $[x]$ denotes the integer part of a real number $x$. Let $p =
e^{2\beta_1}/(1+e^{2\beta_1})$. Then
\[
\lim_{\beta_2\to-\infty} \sup_{\widetilde{f}\in\WF^*(\beta
_2)}\delta_\Box(
\widetilde{f}, p\widetilde{g}) = 0
\]
and
\[
\lim_{\beta_2\to-\infty} \psi(\beta_2)= \frac{(\chi
(H)-2)}{2(\chi(H)-1)}\log
\frac{1}{1-p}.
\]
\end{theorem}
Intuitively, the above result means that if $\beta_2$ is a large
negative number and $n$ is large, then an exponential random graph
$G_n$ with sufficient statistic $T$ looks roughly like a complete
$(\chi(H)-1)$-equipartite graph with $1-p$ fraction of edges randomly
deleted, where $p=e^{2\beta_1}/(1+e^{2\beta_1})$. In particular, if
$H$ is bipartite, then $G_n$ must be very sparse, since a
$1$-equipartite graph has no edges. Figure~\ref{bip} gives a simulation
result for the triangle model with large negative $\beta_2$.

%
%
\begin{figure}

\includegraphics{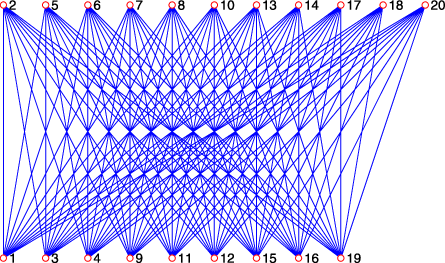}

\caption{A simulated realization of the exponential random graph model
on $20$ nodes with edges and triangles as sufficient statistics, where
$\beta_1=120$ and $\beta_2=-400$. (Picture by Sukhada Fadnavis. Gibbs
sampling used in simulations. We are unsure about the convergence of
the MCMC algorithm used to generate these grapgs, but the structure of
the simulated graphs match the predictions of Theorem \protect\ref{erdos}.)}
\label{bip}
\end{figure}

Theorem \ref{erdos} is closely related to the Erd\H{o}s--Stone theorem
\cite
{erdosstone} from extremal graph theory (or equivalently, Tur\'an's
theorem in the case of triangles as in the work of Fadnavis). Indeed,
it may be possible to prove some parts of our theorem using the Erd\H
{o}s--Stone theorem, but we prefer a bare-hands argument given in
Section \ref{secproof}. Due to this connection with extremal graph
theory, we refer to behavior of the graph in the ``large negative $\beta
_2$'' domain as \textit{extremal behavior}. 


\section{Proofs}\label{secproof}

\mbox{}

\begin{pf*}{Proof of Theorem \ref{soln}}
For each Borel set $\widetilde{A}\subseteq\mmm$ and each $n$,
define
\[
\widetilde{A}_n:= \{\tilh\in\widetilde{A}\dvtx \tilh= \widetilde{G}
\mbox{ for some } G\in\mg_n\}.
\]
Let $\pp_{n,p}$ be the Erd\H{o}s--R\'enyi measure defined in Section
\ref{sec2d}.
Note that $\widetilde{A}_n$ is a finite set and
\[
|\widetilde{A}_n| = 2^{n(n-1)/2}\pp_{n,1/2}(
\widetilde{A}_n) = 2^{n(n-1)/2} \pp_{n,1/2}(\widetilde{A}).
\]
Thus, if $\tilf$ is a closed subset of $\mmm$ then by Theorem \ref{cvmain}
%
%
\begin{eqnarray}\label{count1}
\limsup_{n\to\infty}\frac{\log|\tilf_n|}{n^2} &\le& \frac{\log
2}{2} - \inf
_{\tilh\in\tilf} I_{1/2}(\tilh)
\nonumber\\[-8pt]\\[-8pt]
&=& -\inf_{\tilh\in\tilf} I(\tilh).\nonumber
\end{eqnarray}
Similarly if $\widetilde{U}$ is an open subset of $\mmm$,
%
%
\begin{equation}
\label{count2} \liminf_{n\to\infty}\frac{\log|\widetilde{U}_n|}{n^2} \ge
-\inf
_{\tilh\in\widetilde{U}} I(\tilh).
\end{equation}
Fix $\ep> 0$. Since $T$ is a bounded function, there is a finite set
$R$ such that the intervals $\{(a, a+\ep)\dvtx a\in R\}$ cover the range
of $T$. For each $a\in R$, let $\tilf^a:= T^{-1}([a, a+\ep])$. By the
continuity of $T$, each $\tilf^a$ is closed. Now,
\[
e^{n^2\psi_n} \le\sum_{a\in R} e^{n^2(a+\ep)} \bigl|
\tilf^a_n\bigr|\le|R|\sup_{a\in R}
e^{n^2(a+\ep)} \bigl|\tilf^a_n\bigr|.
\]
By (\ref{count1}), this shows that
\[
\limsup_{n\to\infty} \psi_n \le\sup_{a\in R}
\Bigl(a + \ep-\inf_{\tilh\in\tilf^a} I(\tilh) \Bigr).
\]
Each $\tilh\in\tilf^a$ satisfies $T(\tilh) \ge a$. Consequently,
\[
\sup_{\tilh\in\tilf^a} \bigl(T(\tilh)- I(\tilh)\bigr) \ge\sup
_{\tilh\in
\tilf^a} \bigl(a-I(\tilh)\bigr) = a- \inf_{\tilh\in\tilf^a}
I(\tilh).
\]
Substituting this in the earlier display gives
%
%
\begin{eqnarray}\label{upbd}
\limsup_{n\to\infty} \psi_n &\le& \ep+ \sup
_{a\in R} \sup_{\tilh
\in\tilf^a} \bigl(T(\tilh)- I(\tilh)
\bigr)
\nonumber\\[-8pt]\\[-8pt]
&=& \ep+ \sup_{\tilh\in\mmm} \bigl(T(\tilh)-I(\tilh)\bigr).
\nonumber
\end{eqnarray}
For each $a\in R$, let $\tilu^a:= T^{-1}((a, a+\ep))$. By the continuity
of $T$, $\tilu^a$ is an open set. Note that
\[
e^{n^2 \psi_n} \ge\sup_{a\in R} e^{n^2 a} \bigl|
\tilu^a_n\bigr|.
\]
Therefore by (\ref{count2}), for each $a\in R$
\[
\liminf_{n\to\infty} \psi_n \ge a -\inf
_{\tilh\in\tilu^a} I(\tilh).
\]
Each $\tilh\in\tilu^a$ satisfies $T(\tilh)< a+ \ep$. Therefore,
\[
\sup_{\tilh\in\tilu^a} \bigl(T(\tilh)- I(\tilh)\bigr) \le\sup
_{\tilh\in
\tilu^a} \bigl(a+\ep-I(\tilh)\bigr) = a+\ep- \inf
_{\tilh\in\tilu^a} I(\tilh).
\]
Together with the previous display, this shows that
%
%
\begin{eqnarray}\label{lbd}
\liminf_{n\to\infty} \psi_n &\ge& -\ep+ \sup
_{a\in R} \sup_{\tilh
\in\tilu^a} \bigl(T(\tilh)- I(\tilh)
\bigr)
\nonumber\\[-8pt]\\[-8pt]
&=& -\ep+ \sup_{\tilh\in\mmm} \bigl(T(\tilh)-I(\tilh)\bigr).
\nonumber
\end{eqnarray}
Since $\ep$ is arbitrary in (\ref{upbd}) and (\ref{lbd}), this
completes the proof.
\end{pf*}

\begin{pf*}{Proof of Theorem \ref{limitbehave}}
Take any $\eta> 0$. Let
\[
\widetilde{A}:= \bigl\{\tilh\dvtx \delta_\Box\bigl(\tilh, \tilf^*\bigr)
\ge\eta\bigr\}.
\]
It is easy to see that $\widetilde{A}$ is a closed set. By
compactness of $\mmm$ and $\tilf^*$, and upper semi-continuity of $T-I$,
it follows that
\[
2\gamma:= \sup_{\tilh\in\mmm} \bigl(T(\tilh)-I(\tilh)\bigr) - \sup
_{\tilh
\in\widetilde{A}} \bigl(T(\tilh)-I(\tilh)\bigr) > 0.
\]
Choose $\ep= \gamma$ and define $\tilf^a$ and $R$ as in the proof of
Theorem \ref{soln}. Let $\widetilde{A}{}^a:= \widetilde{A}\cap\tilf^a$.
Then
\[
\pp(G_n\in\widetilde{A}) \le e^{-n^2\psi_n}\sum
_{a\in R} e^{n^2(a+\ep)}\bigl|\widetilde{A}{}^a_n\bigr|
\le e^{-n^2\psi_n} |R|\sup_{a\in
R}e^{n^2(a+\ep)}\bigl|
\widetilde{A}{}^a_n\bigr|.
\]
While bounding the last term above, it can be assumed without loss of
generality that $\widetilde{A}{}^a$ is nonempty for each $a\in R$, for
the other $a$'s can be dropped without upsetting the bound. By
(\ref{count1}) and Theorem \ref{soln} (noting that $\widetilde{A}{}^a$
is compact), the above display gives
\[
\limsup_{n\to\infty} \frac{\log\pp(G_n\in\widetilde{A})}{n^2} \le\sup
_{a\in R}
\Bigl(a + \ep-\inf_{\tilh\in\widetilde{A}{}^a} I(\tilh) \Bigr) - \sup
_{\tilh\in\mmm} \bigl(T(\tilh)-I(\tilh)\bigr).
\]
Each $\tilh\in\widetilde{A}{}^a$ satisfies $T(\tilh) \ge a$. Consequently,
\[
\sup_{\tilh\in\widetilde{A}{}^a} \bigl(T(\tilh)- I(\tilh)\bigr) \ge\sup
_{\tilh\in\widetilde{A}{}^a} \bigl(a-I(\tilh)\bigr) = a- \inf_{\tilh\in
\widetilde{A}{}^a}
I(\tilh).
\]
Substituting this in the earlier display gives
\begin{eqnarray*}
\limsup_{n\to\infty} \frac{\log\pp(G_n\in\widetilde{A})}{n^2} &\le& \ep+
\sup
_{a\in R} \sup_{\tilh\in\widetilde{A}{}^a} \bigl(T(\tilh)- I(\tilh)
\bigr) - \sup_{\tilh\in\mmm} \bigl(T(\tilh)-I(\tilh)\bigr)
\\
&=& \ep+ \sup_{\tilh\in\widetilde{A}} \bigl(T(\tilh)-I(\tilh)\bigr) -
\sup
_{\tilh\in\mmm} \bigl(T(\tilh)-I(\tilh)\bigr).
\\
&=& \ep-2\gamma= -\gamma.
\end{eqnarray*}
This completes the proof.
\end{pf*}

\begin{pf*}{Proof of Theorem \ref{specialnorm}}
By Theorem \ref{soln},
%
%
\begin{equation}
\label{xvar} \lim_{n\to\infty}\psi_n=\sup
_{h\in\mm}\bigl(T(h)-I(h)\bigr).
\end{equation}
By H\"older's inequality,
\[
t(H_i,h)\leq\iint_{[0,1]^2} h(x,y)^{e(H_i)}\,dx\,dy.
\]
Thus, by the nonnegativity of $\beta_2,\ldots,\beta_k$,
\begin{eqnarray*}
T(h)&\leq&\beta_1t(H_1,h)+\sum
_{i=2}^k\beta_i\iint
_{[0,1]^2}h(x,y)^{e(H_i)}\,dx\,dy
\\
&=&\iint_{[0,1]^2}\sum_{i=1}^k
\beta_ih(x,y)^{e(H_i)}\,dx\,dy.
\end{eqnarray*}
On the other hand, the inequality in the above display becomes an
equality if $h$ is a constant function. Therefore, if $u$ is a point
in $[0,1]$ that maximizes
\[
\sum_{i=1}^k\beta_iu^{e(H_i)}-I(u),
\]
then the constant function $h(x,y)\equiv u$ solves the variational
problem (\ref{xvar}). To see that constant functions are the only
solutions, assume that there is at least one $i$ such that the graph
$H_i$ has at least one vertex with two or more neighbors. The above
steps show that if $h$ is a maximizer, then for each $i$,
%
%
\begin{equation}
\label{equality} t(H_i,h)=\iint_{[0,1]^2}h(x,y)^{e(H_i)}
\,dx\,dy.
\end{equation}
In other words, equality holds in H\"older's inequality. Suppose that
$H_i$ has vertex set $\{1,2,\ldots, k\}$ and vertices $2$ and $3$ are
both neighbors of $1$ in $H_i$. Recall that
\[
t(H_i,h) = \int_{[0,1]^k} \prod
_{\{j,l\}\in E(H_i)} h(x_j, x_l) \,dx_1
\cdots dx_k.
\]
In particular, the integrand contains the product
$h(x_1,x_2)h(x_1,x_3)$. From this and the criterion for equality in
H\"older's inequality, it follows that $h(x_1,x_2) = h(x_1,x_3)$ for
almost every $(x_1,x_2,x_3)$. Using the symmetry of $h$ one can now
easily conclude that $h$ is almost everywhere a constant function.

If the condition does not hold, then each $H_i$ is a union of
vertex-disjoint edges. Assume that some $H_i$ has more than one edge.
Then again by (\ref{equality}) it follows that $h$ must be a constant
function.

Finally, if each $H_i$ is just a single edge, then the maximization
problem (\ref{xvar}) can be explicitly solved and the solutions are
all constant functions.
\end{pf*}

\begin{pf*}{Proof of Theorem \ref{specialbehave}}
The assertions about graph limits in this theorem are direct
consequences of Theorems \ref{limitbehave} and \ref{specialnorm}.
Since $\sum_{i=1}^k \beta_i u^{e(H_i)}$ is a polynomial function of
$u$ and $I(u)$ is sufficiently well-behaved, showing that $U$ is a
finite set is a simple analytical exercise.
\end{pf*}

\begin{pf*}{Proof of Theorem \ref{degen}}
Fix $\beta_1 < 0$ such that $c_1< c_2$. As a preliminary step, let us
prove that for any $\beta_2 >0$,
%
%
\begin{equation}
\label{prelim} \lim_{n\to\infty}\pp\bigl(f(G_n)
\in(c_1, c_2)\bigr)=0.
\end{equation}
Fix $\beta_2 > 0$. Let $u$ be any maximizer of $\ell$. Then
by Theorem \ref{specialbehave}, it suffices to prove that either
$u< e^{\beta_1}/(1+e^{\beta_1})$ or $u > 1+1/2\beta_1$. This is
proved as follows. Define a function $g\dvtx[0,1]\to\rr$ as
\[
g(v):= \ell\bigl(v^{1/3}\bigr).
\]
Then $\ell$ is maximized at $u$ if and only if $g$ is maximized at
$u^3$. Since $\ell$ is a bounded continuous function and $\ell'(0)=
\infty$, $\ell'(1)=-\infty$, $\ell$ cannot be maximized at $0$ or $1$.
Therefore, the same is true for $g$. Let $v$ be a point in $(0,1)$ at
which $g$ is maximized. Then $g''(v)\le0$. A simple computation shows
that
\[
g''(v) = \frac1{9v^{5/3}} \biggl(-2
\beta_1 + \log\frac
{v^{1/3}}{1-v^{1/3}}- \frac1{2(1-v^{1/3}
)} \biggr).
\]
Thus, $g''(v)\le0$ only if
\[
\log\frac{v^{1/3}}{1-v^{1/3}} \le\beta_1 \quad\mbox{or}\quad-\frac
{1}{2(1-v^{1/3})} \le
\beta_1.
\]
This shows that a maximizer $u$ of $\ell$ must satisfy $u\le c_1$ or
$u \ge c_2$. Now, if $u=c_1$, then $u < c_2$, and therefore the above
computations show that $g''(v) > 0$, where $v = u^3$. Similarly, if
$u=c_2$ then $u > c_1$ and again $g''(v)>0$. Thus, we have proved that
$u < c_1$ or $u> c_2$.
By Theorem \ref{limitbehave}, this completes the proof of (\ref
{prelim}) when
$\beta_2> 0$.

Now notice that as $\beta_2\ra\infty$, $\sup_{u\le a}\ell(u) \sim
\beta_2 a^3$ for any fixed $a \le1$. This shows that as $\beta_2\ra
\infty$, any maximizer of $\ell$ must eventually be larger than
$1+1/2\beta_1$. Therefore, for sufficiently large $\beta_2$,
%
%
\begin{equation}
\label{prelim1} \lim_{n \ra\infty} \pp\bigl(f(G_n) <
c_2\bigr) = 0.
\end{equation}
Next consider the
case $\beta_2\le0$. Let $\tilf^*$ be the set of maximizers of
$T(\tilh)-I(\tilh)$. Take any $\tilh\in\tilf^*$ and let $h$ be a
representative element of $\tilh$. Let $p= e^{2\beta_1}/(1+e^{2\beta_1})$.
An easy verification shows that
\[
T(h)-I(h) = \beta_2t(H_2, h) - I_p(h) -
\tfrac12\log(1-p),
\]
where $I_p(h)$ is defined as in (\ref{ipdef2}). Define a new function
\[
h_1(x,y):= \min\bigl\{h(x,y), p\bigr\}.
\]
Since the function $I_p$ defined in (\ref{ipdef1}) is minimized at
$p$, it follows that for all $x,y\in[0,1]$, $I_p(h_1(x,y))\le
I_p(h(x,y))$. Consequently, $I_p(h_1)\le I_p(h)$. Again, since
$\beta_2 \le0$ and $h_1\le h$ everywhere, $\beta_2 t(H_2, h_1) \ge
\beta_2 t(H_2, h)$. Combining these observations, we see that $T(h_1)
- I(h_1) \ge T(h)-I(h)$. Since $h$ maximizes $T-I$ it follows that
equality must hold at every step in the above deductions, from which
it is easy to conclude that $h=h_1$ a.e. In other words, $h(x,y)\le p$
a.e. This is true for every $\tilh\in\tilf^*$. Since $p < c_1$, the
above deduction coupled with Theorem \ref{limitbehave} proves that when
$\beta
_2\le0$,
%
%
\begin{equation}
\label{prelim2} \lim_{n\ra\infty} \pp\bigl(f(G_n) >
c_1\bigr) = 0.
\end{equation}
Recalling that $\beta_1$ is fixed, define
\[
a_n(\beta_2):= \pp\bigl(f(G_n)>
c_1\bigr),\qquad b_n(\beta_2):= \pp
\bigl(f(G_n) < c_2\bigr).
\]
Let $A_n$ and $B_n$ denote the events in brackets in the above display.
A simple computation shows that
\[
a_n'(\beta_2) = \frac{6}{n}\cov
\bigl(1_{A_n}, \Delta(G_n)\bigr) \quad\mbox{and}\quad
b_n'(\beta_2) = \frac{6}{n}\cov
\bigl(1_{B_n}, \Delta(G_n)\bigr),
\]
where $\Delta(G_n)$ is the number of triangles in $G_n$.
As noted at the end of Section~\ref{sec4}, the exponential random
graph model with $\beta_2\ge0$ satisfies the FKG criterion \cite
{fkg71}. Therefore, the above identities show that on the nonnegative
axis, $a_n$ is a nondecreasing function and $b_n$ is a nonincreasing function.

Let $q_1:= \sup\{x\in\rr\dvtx \lim_{n\ra\infty} a_n(x) = 0\}$. By
equation (\ref{prelim1}), $q_1<\infty$ and by equation (\ref
{prelim2}) $q_1 \ge0$. Similarly, if $q_2:= \inf\{x\in\rr\dvtx \lim
_{n\ra
\infty} b_n(x) = 0\}$, then $0\le q_2 < \infty$. Also,
clearly, $q_1 \le q_2$ since $a_n+ b_n \ge1$ everywhere. We claim that
\mbox{$q_1 = q_2$}. This would complete the proof by the monotonicity of $a_n$
and $b_n$.

To prove that $q_1 = q_2$, suppose not. Then $q_1 < q_2$. Then for any
$\beta_2 \in(q_1, q_2)$, $\limsup a_n(\beta_2) > 0$ and $\limsup
b_n(\beta_2) > 0$. Now,
\[
0\le a_n(\beta_2)+b_n(\beta_2)-1
= \pp\bigl(f(G_n) \in(c_1, c_2)\bigr).
\]
Therefore by (\ref{prelim}),
\[
\lim_{n\ra\infty}\bigl(a_n(\beta_2) +
b_n(\beta_2) - 1\bigr) = 0.
\]
Thus, for any $\beta_2 \in(q_1, q_2)$, $\limsup(1-b_n(\beta_2)) >
0$. By Theorem \ref{specialbehave}, this implies that the function $\ell
$ has
a maximum in $[c_2, 1]$. Similarly, for any $\beta_2\in(q_1, q_2)$,
$\limsup(1-a_n(\beta_2)) > 0$ and therefore the function $\ell$ has
a maximum in $[0, c_1]$. Now fix $q_1 < \beta_2<\widetilde{\beta}_2<
q_2$, and let $\ell$ and $\widetilde{\ell}$ denote the two $\ell
$-functions corresponding to $\beta_2$ and $\widetilde{\beta}_2$,
respectively. That is,
\[
\ell(u) = \beta_1 u +\beta_2 u^3 - I(u),\qquad
\widetilde{\ell}(u) = \beta_1 u + \widetilde{\beta}_2u^3
- I(u).
\]
By the above argument, $\ell$ attains its maximum at some point $u_1
\in[0,c_1]$ and at some point $u_2\in[c_2,1]$. (There may be other
maxima, but that is irrelevant for us.) Note that
\[
\max_{u\le c_1} \widetilde{\ell}(u) = \max_{u\le c_1}
\bigl(\ell(u) + (\widetilde{\beta}_2-\beta_2) u^3
\bigr) \le\ell(u_1) + (\widetilde{\beta}_2-
\beta_2) c_1^3.
\]
On the other hand
\[
\max_{u\ge c_2} \widetilde{\ell}(u) \ge\widetilde{\ell}(u_2) =
\ell(u_2) + (\widetilde{\beta}_2-\beta_2)
u_2^3 \ge\ell(u_2) + (\widetilde{\beta
}_2-\beta_2) c_2^3.
\]
Since $\ell(u_1)=\ell(u_2)$, $\widetilde{\beta}_2 > \beta_2$ and $c_2
> c_1$, this shows that
\[
\max_{u\le c_1} \widetilde{\ell}(u) < \max_{u\ge c_2}
\widetilde{\ell}(u),
\]
contradicting our previous deduction that $\widetilde{\ell}$ has maxima
in both $[0,c_1]$ and $[c_2,1]$. This proves that $q_1 = q_2$.
\end{pf*}

\begin{pf*}{Proof of Theorem \ref{euler}}
Let $g$ be a symmetric bounded measurable function from $[0,1]$ into
$\rr$. For each $u\in\rr$, let
\[
h_u(x,y):= h(x,y) + ug(x,y).
\]
Then $h_u$ is a symmetric bounded measurable function from $[0,1]$ into
$\rr$. First, suppose that $h$ is bounded away from $0$ and $1$. Then
$h_u\in\mm$ for every $u$ sufficiently small in magnitude. Since $h$
maximizes $T(h)-I(h)$ among all elements of $\mm$, therefore under the
above assumption, for all $u$ sufficiently close to zero,
\[
T(h_u)-I(h_u)\le T(h)-I(h).
\]
In particular,
%
%
\begin{equation}
\label{eulercond} \frac{d}{du}\bigl(T(h_u)-I(h_u)
\bigr) \bigg|_{u=0} = 0.
\end{equation}
It is easy to check that $T(h_u)-I(h_u)$ is differentiable in $u$ for
any $h$ and $g$. In particular, the derivative is given by
\[
\frac{d}{du}\bigl(T(h_u)-I(h_u)\bigr) = \sum
_{i=1}^k \beta_i\,
\frac{d}{du} t(H_i, h_u) - \frac{d}{du}I(h_u).
\]
Now,
\begin{eqnarray*}
\frac{d}{du}I(h_u) &=& \iint\frac{d}{du}I\bigl(h(x,y) +
u g(x,y)\bigr)\,dy\,dx
\\
&=& \frac12\iint g(x,y)\log\frac{h_u(x,y)}{1-h_u(x,y)}\,dy\,dx.
\end{eqnarray*}
Consequently,
\[
\frac{d}{du}I(h_u) \bigg|_{u=0} = \frac12\iint g(x,y)
\log\frac
{h(x,y)}{1-h(x,y)}\,dy\,dx.
\]
Next, note that
\begin{eqnarray*}
&&\frac{d}{du}t(H_i,h_u)
\\
&&\qquad=\int_{[0,1]^{V(H)}}\sum_{(r,s)\in E(H_i)}g(x_r,x_s)
\mathop{\prod_{(r',s')\in E(H_i)}}_{(r',s')\neq(r,s)}
h_u(x_{r'},x_{s'})\prod
_{v\in V(H)}\,dx_v
\\
&&\qquad=\iint g(x,y)\Delta_{H_i}h_u(x,y)\,dy\,dx.
\end{eqnarray*}
Combining the above computations and (\ref{eulercond}), we see that
for any symmetric bounded measurable $g\dvtx[0,1]\to\rr$,
\[
\iint g(x,y) \Biggl(\sum_{i=1}^k
\beta_i\Delta_{H_i}h(x,y)-\frac12\log\frac{h(x,y)}{1-h(x,y)}
\Biggr)\,dy\,dx=0.
\]
Taking $g(x,y)$ equal to the function within the brackets (which is
bounded since $h$ is assumed to be bounded away from $0$ and $1$), the
conclusion of the theorem follows.

Now note that the theorem was proved under the assumption that $h$ is
bounded away from $0$ and $1$. We claim that this is true for any $h$
that maximizes $T(h)-I(h)$. To prove this claim, take any such $h$.
Fix $p \in(0,1)$. For each $u \in[0,1]$, let
\[
h_{p,u}(x,y):= h(x,y) + u\bigl(p-h(x,y)\bigr)_+. 
\]
In other words, $h_{p,u}$ is simply $h_u$ with $g= (p-h)_+$. Then
certainly, $h_{p,u}$ is a symmetric bounded measurable function
from $[0,1]^2$ into $[0,1]$. Note that
\[
\frac{d}{du} h_{p,u}(x,y) = \bigl(p-h(x,y)\bigr)_+.
\]
Using this, an easy computation as above shows that
\begin{eqnarray*}
&&\frac{d}{du} \bigl(T(h_{p,u})-I(h_{p,u}) \bigr)
\bigg|_{u=0}
\\
&&\qquad=\iint\Biggl(\sum_{i=1}^k
\beta_i\Delta_{H_i}h(x,y)-\frac12\log\frac{h(x,y)}{1-h(x,y)}
\Biggr) \bigl(p-h(x,y)\bigr)_+\,dy\,dx
\\
&&\qquad\geq\iint\biggl(-C-\frac12\log\frac
{h(x,y)}{1-h(x,y)} \biggr) \bigl(p-h(x,y)
\bigr)_+\,dy\,dx,
\end{eqnarray*}
where $C$ is a positive constant depending only on $\beta_1,\ldots,
\beta_k$ and $H_1,\ldots, H_k$ (and not on $p$ or $h$). When
$h(x,y)=0$, the integrand is interpreted as $\infty$, and when
$h(x,y)=1$, the integrand is interpreted as
$0$. 

Now, if $p$ is so small that
\[
-C - \frac{1}{2}\log\frac{p}{1-p} > 0,
\]
then the previous display proves that the derivative of
$T(h_{p,u})-I(h_{p,u})$ with respect to $u$ is strictly positive at
$u=0$ if $h < p$ on a set of positive Lebesgue measure. Hence, $h$
cannot be a maximizer of $T-I$ unless $h\ge p$ almost everywhere. This
proves that any maximizer of $T-I$ must be bounded away from zero.
A~similar argument with $g= -(h-p)_+$ shows that it must be bounded
away from $1$, and hence completes the proof of the theorem.
\end{pf*}

\begin{pf*}{Proof of Theorem \ref{contraction}}
It suffices to prove that the maximizer of $T(h)-I(h)$ as $h$ varies
over $\mm$ is unique. This is because: if $h$ is a maximizer, then
so is $h_\sigma(x,y):= h(\sigma x, \sigma y)$ for any measure
preserving bijection $\sigma\dvtx[0,1]\to[0,1]$. The only functions that
are invariant under such transforms are constant functions. 

Let $\Delta_H$ be the operator defined in Section \ref{sec62}. Let $\|
\cdot\|_\infty$ denote the $L^\infty$ norm on $\mm$ (i.e.,
the essential supremum of the absolute value). Let $h$ and $g$ be
two maximizers of $T-I$. For any finite simple graph $H$, a simple
computation shows that
\begin{eqnarray*}
\|\Delta_H h - \Delta_H g\|_\infty&\le&\sum
_{(r,s)\in E(H)} \| \Delta_{H, r,s} h -
\Delta_{H,r,s} g\|_\infty
\\
&\le& e(H) \bigl(e(H) - 1\bigr) \|h - g\|_\infty.
\end{eqnarray*}
Using the above inequality, Theorem \ref{euler} and the inequality
\[
\biggl|\frac{e^x}{1+e^x} - \frac{e^y}{1+e^y} \biggr|\le\frac{|x-y|}{4}
\]
(easily proved by the mean value theorem) it follows that for almost
all $x,y$,
\begin{eqnarray*}
\bigl|h(x,y)-g(x,y)\bigr| &=& \biggl| \frac{e^{2\sum_{i=1}^k \beta_i \Delta_{H_i}
h(x,y)}}{1 + e^{2\sum_{i=1}^k \beta_i \Delta_{H_i} h(x,y)}} - \frac
{e^{2\sum_{i=1}^k \beta_i \Delta_{H_i} g(x,y)}}{1 + e^{2\sum
_{i=1}^k \beta_i \Delta_{H_i} g(x,y)}} \biggr|
\\
&\le&\frac{1}{2}\sum_{i=1}^k |
\beta_i| \|\Delta_{H_i} h - \Delta_{H_i} g
\|_\infty
\\
&\le&\frac{1}{2} \|h-g\|_\infty\sum_{i=1}^k
|\beta_i| e(H_i) \bigl(e(H_i)-1\bigr).
\end{eqnarray*}
If the coefficient of $\|h-g\|_\infty$ in the last expression is
strictly less than $1$, it follows that $h$ must be equal to $g$ a.e.
\end{pf*}

\begin{pf*}{Proof of Theorem \ref{symmbreaking}}
Fix $\beta_1$. Let $p = e^{2\beta_1}/(1+e^{2\beta_1})$ and $\gamma:=
-\beta_2$, so that for any $h\in\mm$,
\[
T(h)-I(h) = -\gamma t(H_2, h) - I_p(h) -
\tfrac{1}{2}\log(1-p).
\]
Assume without loss of generality that $\beta_2< 0$. Suppose $u$ is a
constant such that the function $h(x,y)\equiv u$ maximizes
$T(h)-I(h)$, that is, minimizes $\gamma t(H_2,h) + I_p(h)$. Note that
\[
\gamma t(H_2, h) + I_p(h) = \gamma u^3 +
I_p(u).
\]
Clearly, the definition of $u$ implies that $\gamma u^3 + I_p(u) \le
\gamma x^3 + I_p(x)$ for all $x\in[0,1]$. This implies that $u$ must
be in $(0,1)$, because the derivative of $x\mapsto\gamma x^3 +
I_p(x)$ is $-\infty$ at $0$ and $\infty$ at $1$. Thus,
\[
0=\frac{d}{dx}\bigl(\gamma x^3+I_p(x)\bigr)
\bigg|_{x=u}=3\gamma u^2+\frac12\log\frac{u}{1-u} -
\frac12\log\frac{p}{1-p},
\]
which shows that $u \le c(\gamma)$, where $c(\gamma)$ is a function of
$\gamma$ such that
\[
\lim_{\gamma\to\infty} c(\gamma)=0.
\]
This shows that
%
%
\begin{equation}
\label{cgamma} \lim_{\gamma\to\infty}\min_{0\le x\le1} \bigl(
\gamma x^3+I_p(x) \bigr)=I_p(0)=\frac12\log
\frac1{1-p}.
\end{equation}
Next, let $g$ be the function
\[
g(x,y):= \cases{ 0, &\quad if $x, y$ on same side of $1/2$,
\cr
p, &\quad if not.}
\]
Clearly, for almost all $(x,y,z)$, $g(x,y)g(y,z)g(z,x)=0$. Thus,
$t(H_2, g) = 0$. A~simple computation shows that
\[
I_p(g) = \frac{1}{4}\log\frac{1}{1-p}.
\]
Thus, $\gamma t(H_2,g)+I_p(g) = \frac{1}{4}\log\frac{1}{1-p}$. This
shows that if $\gamma$ is large enough (depending on $p$ and hence
$\beta_1$), then $T-I$ cannot be maximized at a constant function. The
rest of the conclusion follows easily from Theorem \ref{limitbehave} and
the compactness of $\mmm$.
\end{pf*}

\begin{pf*}{Proof of Theorem \ref{pstarcase}}
Take any $h\in\mm$. Note that
\begin{eqnarray*}
t(H_j, h) &=& \int_{[0,1]^j} h(x_1,x_2)h(x_1,x_3)
\cdots h(x_1, x_j)\,dx_1\cdots
dx_j
\\
&=& \int_0^1 M(x)^j \,dx,
\end{eqnarray*}
where
\[
M(x)=\int_0^1 h(x,y)\,dy.
\]
Since $I$ is a convex function,
\[
\int_0^1 I\bigl(h(x,y)\bigr)\,dy \ge I
\bigl(M(x)\bigr)
\]
with equality if and only if $h(x,y)$ is the same for almost all $y$.
Thus, putting
\[
P(u):= \sum_{j=1}^k \beta_j
u^j,
\]
we get
\[
T(h)-I(h) = \int_0^1 P\bigl(M(x)\bigr)\,dx -
I(h)\le\int_0^1 \bigl(P\bigl(M(x)\bigr) - I
\bigl(M(x)\bigr)\bigr) \,dx
\]
with equality if and only if for almost all $x$, (a) $h(x,y)$ is
constant as a function of~$y$, and (b) $M(x)$ equals a value $u^*$ that
maximizes $P(u)-I(u)$. By the symmetry of $h$, the condition (a)
implies that $h$ is constant almost everywhere. The condition (b) gives
the set of possible values of this constant. The rest follows as in the
proofs of Theorems \ref{specialnorm} and \ref{specialbehave}.
\end{pf*}

%
%
\begin{lem}\label{chromlmm}
Let $r$ be any integer $\ge\chi(H)$. Let $K_r$ be the complete graph
on $r$ vertices. Then for any symmetric measurable $h\dvtx[0,1]^2 \ra\{
0,1\}$, if $t(K_r, h) > 0$ then $t(H, h) > 0$.
\end{lem}
\begin{pf}
Let $h_n(x,y)$ be the average value of $h$ in the dyadic square of
width $2^{-n}$ containing the point $(x,y)$. A standard martingale
argument implies that the sequence of functions $\{h_n\}_{n\ge1}$
converges to $h$ almost everywhere. For any positive integer $u$, let
$K_r^u$ denote the complete $r$-partite graph on $ru$ vertices, where
each partition consists of $u$ vertices (so that $K_r^1 = K_r$). Since
$r\ge\chi(H)$, it is easy to see that there exists $u$ so large that
$H$ is a subgraph of $K_r^u$ [i.e., $V(H)\subseteq V(K^u_r)$ and $E(H)
\subseteq E(K^u_r)$]. Fix such a $u$.

By the almost everywhere convergence of $h_n$ to $h$ and the assumption
that $t(K_r, h) >0$, there is a set of $r$ distinct points $x_1,\ldots,
x_r\in[0,1]$ that do not lie on the boundary of any dyadic interval,
such that $h(x_i, x_j) > 0$ and $\lim_{n\ra\infty} h_n(x_i, x_j) =
h(x_i, x_j)$ for each $1\le i\ne j\le r$. Since $h$ is $\{0,1\}
$-valued, $h(x_i, x_j) =1$ for each $i\ne j$. Choose $n$ so large that
for each $i\ne j$,
\[
h_n(x_i, x_j) \ge1- \ep,
\]
where $\ep= 1/2r^2u^2 $. Let $(X^s_i)_{1\le i \le r, 1\le s\le u}$
be independent random variables, where $X^s_i$ is uniformly distributed
in the dyadic interval of width $2^{-n}$ containing $x_i$. Then for
each $1\le i\ne j\le r$, $1\le q, s \le u$,
\[
\pp\bigl(h\bigl(X_i^q, X_j^s
\bigr)=1\bigr) = h_n(x_i, x_j)\ge1-\ep.
\]
Therefore,
\[
\pp\bigl(h\bigl(X_i^q, X_j^s
\bigr)=1 \mbox{ for all } 1\le i\ne j\le r, 1\le q, s \le u\bigr)
\ge1-r^2u^2\ep=1/2.
\]
Let $(Y^s_i)_{1\le i \le r, 1\le s\le u}$ be independent random
variables uniformly distributed in $[0,1]$. Conditional on the event
that $Y^s_i$ belongs to the dyadic interval of width $2^{-n}$
containing $x_i$, $Y_i^s$ has the same distribution as $X_i^s$. As a
consequence of the last display, this shows that
\begin{eqnarray*}
t\bigl(K_r^u, h\bigr) &=& \pp\bigl(h
\bigl(Y_i^q, Y_j^s\bigr)=1
\mbox{ for all } 1\le i\ne j\le r, 1\le q, s \le u\bigr)
\\
&\ge&2^{-nru}\pp\bigl(h\bigl(X_i^q,
X_j^s\bigr)=1 \mbox{ for all } 1\le i\ne j\le r, 1\le
q, s \le u\bigr)>0.
\end{eqnarray*}
Since $H$ is a subgraph of $K_r^u$, therefore $t(H, h) > 0$.
\end{pf}
%
%
\begin{theorem}\label{negopt}
Let $g$ be the function defined in (\ref{chrom}). Take any $p\in
(0,1)$. If $f$ is any element of $\mm$ that minimizes $I_p(f)$ among
all $f$ satisfying $t(H,f)=0$, then $\widetilde{f}= p\widetilde{g}$.
\end{theorem}
\begin{pf}
Take any minimizer $f$. (Minimizers exist due to the Lov\'asz--Szegedy
compactness theorem \cite{lovaszszegedy07}, Theorem 5.1, and the lower
semicontinuity of $I_p$.) First, note that $f\le p$ almost everywhere:
if not, then $I_p(f)$ can be decreased by replacing $f$ with $\min\{
f,p\}$, which retains the condition $t(H,f) = 0$.

Next, note that for almost all $x,y$, $f(x,y)=0$ or $p$. If not, then
redefine $f$ to be equal to $p$ wherever $f$ was positive. This
decreases the entropy while retaining the condition $t(H,f)=0$.

Let $h = f/p$. Then $h$ takes value $0$ or $1$ almost everywhere and
$h$ minimizes $I_p(ph)$ among all symmetric measurable $h\dvtx
[0,1]^2\ra\{
0,1\}$ satisfying \mbox{$t(H, h) = 0$}. Equivalently, $h$ maximizes $\iint
h(x,y) \,dx \,dy$ among all symmetric measurable $h\dvtx[0,1]^2 \ra\{
0,1\}$
satisfying $t(H,h)=0$. Our goal is to show that $\wh= \widetilde{g}$.

Let $r:= \chi(H)$. Let $X_0,X_1,X_2,\ldots$ be a sequence of i.i.d.
random variables uniformly distributed in $[0,1]$. Let
\[
\mathcal{R}:= \bigl\{i\dvtx h(X_i, X_j) = 1 \mbox{ for
all } 1\le j< i\bigr\},
\]
and let $R:= |\mathcal{R}|$. Let $\lambda(x):= \int h(x,y) \,dy$, so
that for any given $i$,
\[
\pp\bigl(h(X_i, X_j) = 1 \mbox{ for all } 1\le j< i
\bigr) = \ee\bigl(\lambda(X_i)^{i-1}\bigr) = \ee\bigl(
\lambda(X_0)^{i-1}\bigr).
\]
Thus,
%
%
\begin{eqnarray}\label{ineqs}
\ee(R) &=& \sum_{i=1}^\infty\pp
\bigl(h(X_i, X_j) = 1 \mbox{ for all } 1\le j< i\bigr)
\nonumber
\\
&=& \sum_{i=1}^\infty\ee\bigl(
\lambda(X_0)^{i-1}\bigr)
\\
&\ge&\sum_{i=1}^\infty\bigl(\ee
\lambda(X_0)\bigr)^{i-1} = \frac{1}{1-\ee
\lambda(X_0)} =
\frac{1}{1-\iint h(x,y) \,dx\,dy}.\nonumber
\end{eqnarray}
Let $g$ be the function defined in (\ref{chrom}). Suppose the vertex
set of $H$ is $\{1,\ldots, k\}$ for some integer $k$. If $t(H, g) >
0$, then there exist $x_1,\ldots, x_k$ such that $g(x_i,x_j)=1$
whenever $\{i,j\}$ is an edge in $H$. By the definition of $g$, this
implies that $H$ can be colored by $r-1$ colors so that no two adjacent
vertices receive the same color; since this is false, therefore $t(H,
g)$ must be zero. By the optimality property of $h$, this gives
\[
\iint h(x,y) \,dx \,dy \ge\iint g(x,y) \,dx \,dy = 1- \frac{1}{r-1}.
\]
Therefore by (\ref{ineqs}),
\[
\ee(R) \ge r-1.
\]
Again by Lemma \ref{chromlmm}, $t(K_r, h) = 0$. Therefore, $R\le r-1$ almost
surely. Combined with the above display, this shows that equality must
hold in (\ref{ineqs}) and $R = r-1$ almost surely. In particular, $\ee
(\lambda(X_0)^2 ) = (\ee\lambda(X_0))^2$ and $\ee\lambda(X_0) =
1-1/(r-1)$, which shows that
\[
\lambda(x) = 1-\frac{1}{r-1} \qquad\mbox{a.e.}
\]
For each $x$, let $A(x):= \{y\dvtx h(x,y) = 0\}$. Then $|A(x)| = 1/(r-1)$
a.e., where $|A(x)|$ denotes the Lebesgue measure of $A(x)$.

Define a random graph $G$ on $\{0,1,2,\ldots\}$ by including the edge
$(i,j)$ if and only if $h(X_i,X_j)=1$. Since $t(K_r, h)=0$, $G$ cannot
contain any copy of $K_r$. Thus, with probability $1$, $h(X_0, X_i)=0$
for some $i\in\mathcal{R}$. In other words, $\bigcup_{i\in\mathcal
{R}} A(X_i)$ cover almost all of $[0,1]$. Again, $|A(X_i)|= 1/(r-1)$
for all $i\in\mathcal{R}$ and $|\mathcal{R}|= r-1$ almost surely.
All this together imply that with probability $1$, $A(X_i)\cap A(X_j)$
has Lebesgue measure zero for all $i\ne j\in\mathcal{R}$, since
\[
\sum_{i, j\in\mathcal{R}, i<j} \bigl|A(X_i)\cap
A(X_j)\bigr| \le\sum_{i\in
\mathcal{R}}
\bigl|A(X_i)\bigr| - \biggl|\bigcup_{i\in\mathcal{R}}
A(X_i) \biggr| =0.
\]
Let $Y_1,Y_2,\ldots$ and $Z_1,Z_2, \ldots$ be i.i.d. random
variables uniformly distributed in $[0,1]$, that are independent of the
sequence $X_1,X_2,\ldots\,$. Since $t(K_r, h)=0$, with probability $1$
there cannot exist $l$ and a set $B$ of integers of size $r-2$ such
that $h(Y_l, X_i)= h(Z_l, X_i)=1$ for all $i\in B$, $h(X_i, X_j)=1$ for
all $i\ne j \in B$, and $h(Y_l, Z_l)=1$.

Now fix a realization of $X_1, X_2, \ldots\,$. This fixes the set
$\mathcal{R}$. Take any $i\in\mathcal{R}$. Let $I$ be the smallest
integer such that both $Y_I$ and $Z_I$ are in $A(X_i)$. Clearly $Y_I$
and $Z_I$ are independent and uniformly distributed in $A(X_i)$,
conditional on the sequence $X_1,X_2,\ldots$ and our choice of $i\in
\mathcal{R}$. By the observation from the preceding paragraph, $h(Y_I,
Z_I)=0$ with probability $1$, since the set $\mathcal{R}\setminus\{
i\}$ serves the role of~$B$.

This shows that given $X_1, X_2,\ldots\,$, the sets $A(X_i)$ have the
property that for almost all $y,z\in A(X_i)$, $h(y,z)=0$. Since
$\lambda(x)=1 - 1/(r-1)$ a.e. and $|A(X_i)|=1/(r-1)$, this shows that
for almost all $y\in A(X_i)$ and almost all $z\notin A(X_i)$, $h(y,z)=1$.

The properties of $(A(X_i))_{i\in\mathcal{R}}$ that we established
can be summarized as follows: the sets $A(X_i)$ are disjoint up to
errors of measure zero; each $A(X_i)$ has Lebesgue measure $1/(r-1)$
and together they cover the whole of $[0,1]$; for almost all $y,z\in
[0,1]$, $h(y,z)=0$ if they belong to the same $A(X_i)$, and $h(y,z)=1$
if $y\in A(X_i)$ and $z\in A(X_j)$ for some $i\ne j$. These properties
immediately show that $h$ is the same as the function $g$ up to a
rearrangement; the formal argument can be completed as follows.

Given $X_1,X_2,\ldots\,$, let $u\dvtx[0,1]\ra[0,1]$ be the map defined as
\[
u(x):= \mbox{minimum $i\in\mathcal{R}$ such that $x\in
A(X_i)$}.
\]
Note that with probability $1$, for almost all $x$ there is a unique
$i\in\mathcal{R}$ such that $x\in A(X_i)$. Let $\sigma\dvtx[0,1]\ra[0,1]$
be a measure-preserving bijection such that $x\mapsto u(\sigma x)$ is a
nonincreasing (we omit the construction). Then $\sigma$ maps the
intervals $[0,1/(r-1)]$, $[1/(r-1), 2/(r-1)],\ldots,[(r-2)/(r-1), 1]$
onto the sets $(A(X_i))_{i\in\mathcal{R}}$ up to errors of measure
zero. By the properties of $A(X_i)$ established above, this shows that
$h(\sigma x, \sigma y)$ is the same as $g(x,y)$ up to an error of
measure zero.
\end{pf}

\begin{pf*}{Proof of Theorem \ref{erdos}}
First, note that
\[
T_{\beta_2}(h) - I(h) = \beta_2 t(H, h) - I_p(h) -
\tfrac{1}{2}\log(1-p),
\]
where $p=e^{2\beta_1}/(1+e^{2\beta_1})$. Take a sequence $\beta
_2^{(n)} \ra-\infty$, and for each $n$, let $\wh_n$ be an element of
$\WF^*(\beta_2^{(n)})$. Let $\wh$ be a limit point of $\wh_n$
in $\mmm$. If $t(H, h) > 0$, then by the continuity of the map $t(H,
\cdot)$ and the boundedness of $I_p$,
\[
\lim_{n\to\infty}\psi\bigl(\beta_2^{(n)}\bigr)
=-\infty.
\]
But this is impossible since $\psi(\beta_2^{(n)})$ is uniformly
bounded below, as can be easily seen by considering the function $g$
defined in (\ref{chrom}) as a test function in the variational
problem. Thus, $t(H,h)=0$. If $f$ is a function such that $t(H,f)=0$
and $I_p(f)< I_p(h)$, then for all sufficiently large $n$,
\[
T_{\beta_2^{(n)}}(h_n)-I(h_n) < T_{\beta_2^{(n)}}(f)-I(f)
\]
contradicting the definition of $\WF^*(\beta_2)$. Thus, if $f$ is a
function such that \mbox{$t(H,f)=0$}, then $I_p(f)\ge I_p(h)$. By Theorem \ref
{negopt}, this shows that $\wh=p\widetilde{g}$. The compactness of
$\mmm$ now proves the first part of the theorem.

For the second part, first note that
\begin{eqnarray*}
\liminf_{n\ra\infty} \psi\bigl(\beta_2^{(n)}
\bigr) &\ge&\lim_{n\ra\infty
}\bigl(T_{\beta_2^{(n)}}(pg)-I(pg)\bigr)
= -I_p(pg)-\frac{1}{2}\log(1-p)
\\
&=& \frac{(\chi(H)-2)}{2(\chi(H)-1)}\log\frac{1}{1-p}.
\end{eqnarray*}
Next, note that by the lower-semicontinuity of $I_p$ and the fact that
$\beta_2^{(n)}$ is eventually negative,
\begin{eqnarray*}
\limsup_{n\ra\infty} \psi\bigl(\beta_2^{(n)}
\bigr) &=& \limsup_{n\ra\infty
} \bigl(\beta_2^{(n)}
t(H, h_n) - I_p(h_n)\bigr) -
\frac{1}{2}\log(1-p)
\\
&\le&\limsup_{n\ra\infty} \bigl( - I_p(h_n)
\bigr) - \frac{1}{2}\log(1-p)
\\
&\le&-I_p(pg) - \frac{1}{2}\log(1-p).
\end{eqnarray*}
The proof is complete.
\end{pf*}

\section*{Acknowledgements}

We thank Hans Anderson, Charles Radin, Austen Head, Susan Holmes,
Sumit Mukherjee, and especially Sukhada Fadnavis and Mei Yin for their
substantial help with this paper. We are particularly grateful to the
referee for an exceptionally thorough and useful report and the
improvement to Theorem~\ref{pstarcase}.


%

\printaddresses

\end{document}